\documentclass{amsart}
\usepackage[erhan]{hzmath}

\makeatletter
\@namedef{subjclassname@2020}{%
  \textup{2020} Mathematics Subject Classification}
\makeatother

\begin{document}

\title[Estimations for Graphon]{Non-parametric estimates for graphon mean-field particle systems}

\author[Bayraktar]{Erhan Bayraktar}
\address{%
	Department of Mathematics,
	University of Michigan,
	Ann Arbor, MI 48109.}
\email{erhan@umich.edu  }
\author[Zhou]{Hongyi Zhou}
\address{%
	Department of Mathematics,
	University of Michigan,
	Ann Arbor, MI 48109.}
\email{hongyizh@umich.edu	}

\thanks{This research is supported in part by the National Science Foundation.}
\date{March 9, 2024}

\begin{abstract}
	We consider the graphon mean-field system introduced in the work of Bayraktar, Chakraborty, and Wu.
	It is the large-population limit of a heterogeneously interacting diffusive particle system,
where the interaction is of mean-field type with weights characterized by an underlying graphon function.
	Through observation of continuous-time trajectories within the particle system,
we construct plug-in estimators of the particle density, the drift coefficient, and thus the graphon interaction weights of the mean-field system.  
	Our estimators for the density and drift are direct results of kernel interpolation on the empirical data, and a deconvolution method leads to an estimator of the underlying graphon function.
	We show that, as the number of particles increases, the graphon estimator converges to the true graphon function pointwisely, and as a consequence, in the cut metric.	
	Besides, we conduct a minimax analysis within a particular class of particle systems to justify the pointwise optimality of the density and drift estimators.
\end{abstract}

\maketitle

\tableofcontents

\section{Introduction}\label{s:intro}

We study a statistical method to estimate the interaction strength in the graphon mean-field interaction particle system introduced in~\cite{BayraktarChakrabortyWu2023}. 
The particles in such a system are characterized by not only the feature vector in a physical space $\R^d$ but also the ``type'' indexed by $I = [0, 1]$. 
The interaction strength between particles of different types is quantified by a graphon function $G: I \times I \to [0, 1]$.

Precisely, the system consists of a family of diffusion processes with dynamics
\begin{align}
	\label{e:graphon-SDE}
	X_u(t) & = X_u(0) + \int_0^t \int_I \int_{\R^d} b(X_u(s), x) G(u, v) \mu_{s,v}(dy) dv ds \\
	\nonumber
	& \qquad \qquad + \int_0^t \sigma(X_u(s)) dB_u(s) \,, \qquad t \ge 0  \,, \qquad u \in I \,,
\end{align}
where $b: \R^d \times \R^d \to \R$ and $\sigma: \R^d \to \R^{d \times d}$ are some Lipschitz functions, $\{X_u(0) \st u \in I\}$ are a collection of independent random variables in $\R^d$ with distributions $\{\mu_{0,u} \st u \in I\}$, and $\{B_u \st u \in I\}$ are {i.i.d.}~$d$-dimensional Brownian motions independent of $\{X_u(0) \st u \in I\}$. 
Here, we are assuming that the interactions between particles only happen in the drift term. 

The main purpose of this study is to estimate the graphon function $G$ by continuous observation of a finite-population system. 
It is shown in~\cite{BayraktarChakrabortyWu2023} that system~\eqref{e:graphon-SDE} is the large-population limit of the following system
\begin{align}
	\label{e:finite-pop-SDE}
	X^n_i(t) & = X_{\frac{i}{n}}(0) + \int_0^t \frac{1}{n} \sum_{j=1}^{n} b(X^n_i(s), X^n_j(s)) g^n_{ij} ds \\
	\nonumber
	& \qquad \qquad + \int_0^t \sigma(X^n_i(s)) dB_{\frac{i}{n}}(s) \,, \qquad t \ge 0 \,, \quad i = 1, \dots, n \,, 
\end{align}
where $g^n_{ij} = G(\frac{i}{n}, \frac{j}{n})$. 
We continuously observe~\eqref{e:finite-pop-SDE} over a (fixed) time horizon $[0, T]$. 
Using empirical data, we construct estimators of the particle density and the drift term of~\eqref{e:graphon-SDE}
and finally an estimator of the graphon function $G$. 
The error of our estimation is well-controlled when the number of particles increases, with proper choices of parameters and under certain conditions.

In this work, we are mainly interested in a model resembling the McKean-Vlasov type, where the drift integrand $b$ takes the form
\begin{equation*}
	b(x, y) = F(x - y) + V(x) \,, \qquad x, y \in \R^d \,,
\end{equation*}
for some sufficiently regular functions $F$ and $V$. 
The function $F$ acts as an interacting force between two particles depending on their relative positions, while the function $V$ accounts for an external force applied to every single particle. 
Also, we consider graphon functions $G$ of the form 
\begin{equation*}
	G(u, v) = g(u - v) \,, \qquad u, v \in I \,,
\end{equation*}
for some function $g: \R \to \R$ with certain regularities. 
When we specialize in this case, the problem boils down to estimating the function $g$.

\subsection{Background}\label{s:background}

The study of classical mean-field systems with homogeneous interaction and the associated parabolic equations in the sense of McKean~\cite{McKean1966} dates back to the 1960s.
The original motivation of this study came from plasma theory in statistical physics (see~\cite{Vlasov1961,Sznitman1991,Kolokoltsov2010} and references therein), and its significance in applied mathematics was well demonstrated throughout the past decades.
Several analytic and probabilistic methods were developed during the period to push-forward the study of mean-field systems (see the references in~\cite{MaestraHoffmann2022}).

However, the early formulation of this problem focuses on the theoretical properties of the systems, which highly rely on precise knowledge of the dynamics of the systems. 
Statistical methods that fit those properties into models with noises were still in shortage until the 21st century. 
A modern formulation came to the stage in the 2010s, when the development of other areas of research led to a high demand for statistical inference models. 
Empirical data from a particle system can be utilized as inferences to estimate the dynamics of the system and thus predict its future behaviors. 
Those new ideas are applied to various application fields, including chemical and biological systems~\cite{BaladronFasoliFaugerasTouboul2012,BurgerCapassoMorale2007,MogilnerEdelstein1999}, economics and finance~\cite{FouqueSun2013,Carmona2021}, collective behaviors~\cite{CanutoFagnaniTilli2012,ChazelleJiuLiWang2017}, etc.  

While the features of particles are usually embedded in a physical space $\R^d$, the situations in the study of modern networks can be more complicated. 
Systems with inhomogeneity contain different types of entities (e.g. social networks~\cite{WassermanFaust1994} and power supplies~\cite{SarkarBhattacharjeeSamantaBhattacharyaSaha2019}), and interactions between two individuals also depend on their types. 
Then the idea of studying sophisticated networks using graphon particle systems begins to draw significant attention. 

A heterogeneous particle system can be embedded into a graph (deterministic or random) graph (see~\cite{Coppini2022,CoppiniDietertGiacomin2019,Delarue2017,DelattreGiacomin2016,DupuisMedvedev2020,OliveiraReis2019}), 
where the interaction strength between two types of particles is quantified by the corresponding edge weight. 
As the number of vertices increases and the graph becomes denser, the interaction strength approaches some bounded symmetric kernel $G: [0,1]^2 \to [0,1]$ called \emph{graphon}. 
In fact, every graphon is the limit of a sequence of finite graphs, which is discussed in Chapter 11 of~\cite{Lovasz2012}. 

In recent years, graphon mean field games have demonstrated their ability to model densely interactive networks in multiple studies, e.g. \cite{CarmonaCooneyGraves2021,CainesHuang2021,PariseOzdaglar2023}.
On the purely theoretical side, several results on the stability and stationarity of graphon mean-field systems have been established in~\cite{BayraktarWu2022,BayraktarWu2023,BayraktarChakrabortyWu2023,BayraktarWuZhang2023}, and the concentration of measures is well studied in~\cite{BayraktarWu2023,BayraktarKim2024}. 
These properties enable the study of the mean-field systems from a statistical inference point of view. 

Statistical inference methods are widely applied to learning the dynamics of interactive systems.
Empirical data in a McKean-Vlasov model can be interpolated using a kernel to obtain estimates of particle density~\cite{MaestraHoffmann2022,BelomestnyPilipauskaitePodolskij2023,AmorinoHeidariPilipauskaitePodolskij2023}.
In particular, the data-driven estimation algorithms in~\cite{MaestraHoffmann2022} automatically choose the best kernel bandwidths among a predetermined, possibly opaque set, 
which ensures pointwise optimality even without explicitly specifying the parameters. 
Estimating the interacting force requires more technical tools, including the deconvolution methods introduced by Johannes in~\cite{Johannes2009}. 
These strategies offer firm technical support in the analysis of interactive systems with unknown driving forces, which admits predictions of the evolution of the systems solely based on empirical data.

\subsection{Our contributions and organization of this paper}\label{s:contribute}

Recall the graphon function $G$ in the mean-field system~\eqref{e:graphon-SDE}.
In Section~\ref{s:model}, we introduce a kernel interpolation method adopted from~\cite{MaestraHoffmann2022,BelomestnyPilipauskaitePodolskij2023}.
We make continuous observations of the $n$-particle system~\eqref{e:finite-pop-SDE} during a finite time interval $[0, T]$.
The empirical data from the finite-population system are then interpolated in both the feature space $\R^d$ and the index space $I$ to produce a pointwise estimator $\hat \mu^n_h(t,u,x)$ of the particle density functions $\mu(t,u,x)$.
A further interpolation in the time variable leads to an estimator $\hat \beta^n_{h, \kappa}(t, u, x)$ for the drift coefficients
\begin{equation*}
	\beta(t,u,x) \defeq \int_I G(u, v) (V(x) + F \ast \mu_{t,v} (x)) dv \,.
\end{equation*} 
Then we apply a deconvolution method introduced in~\cite{Johannes2009} to build a pointwise estimator $\hat G^n_{\vartheta}$ of $G$. 
Here 
\begin{equation*}
	\vartheta = (h_1, h_2, h_3, \kappa_0, \kappa_1, \kappa_2, r, \tilde r)
\end{equation*}
are the parameters associated to the estimators:
$h = (h_1,h_2,h_3) \in \R_+^3$ are the bandwidths of the kernels,
$\kappa = (\kappa_0, \kappa_1, \kappa_2) \in \R_+^3$ are the denominator cutoff factors to prohibit fraction blowups,
and $r, \tilde r > 0$ are cutoff radius.
We will explain them with more details in Section~\ref{s:model}. 
We show in Section~\ref{s:main} that there exists a sequence $(\vartheta^n)_{n \in \N}$ of the parameters such that
\begin{equation*}
	\lim_{n \to \infty} \E \abs{\hat G^n_{\vartheta_n} (u_0, v_0) - G(u_0,v_0)}^2 = 0 \,,
\end{equation*}
subject to the regularity conditions.

\vspace{10pt}

We will disclose the particular settings of our problem in Section~\ref{s:model}.
These include the continuity and integrability of the coefficients $F, V, G$ and the initial data $\mu_{0,u}$. 
Then we define the kernel-interpolated estimators $\hat \mu^n_h$ and $\hat \beta^n_{h, \kappa}$, with free choices of the kernel bandwidth vector $h$ and cutoff factors $\kappa$. 
It is worth noticing that the bandwidths of our estimators are fixed throughout the algorithm,
whereas the data-driven Goldenshluger-Lepski estimators applied in~\cite{MaestraHoffmann2022} make dynamic choices of bandwidths from a pre-determined finite set of candidates.
The pre-determined set can be invisible to the users, and the algorithm automatically selects the best candidate to give as output an estimation.
Such algorithm attains the optimal pointwise oracle estimations without precise knowledge of the system's continuity property and does not lose too much efficiency for each tuple of plug-in arguments. 
However, the convergence of our estimator $\hat G^n_\vartheta$ depends on the total $L^2$-errors of the plug-in estimators (instead of the pointwise errors), 
so it becomes more beneficial to fix the bandwidths all along. 
The minimax analysis in Section~\ref{s:minimax} shows that our estimator $\hat \mu^n_h$ is still pointwisely optimal when given enough information. 

We will present upper bounds on the errors of the pointwise estimators in Section~\ref{s:main}, with proofs in Section~\ref{s:estimates-proof}.
The main idea behind the proofs are the stability of the mean-field systems and the concentration of particle density.
We make a direct connection between the (observed) finite-population system~\eqref{e:finite-pop-SDE} and the intrinsic graphon mean field system~\eqref{e:graphon-SDE} by the next inequality step.
With particle density $\mu$, for example, we have
\begin{equation}
	\label{e:telescoping}
	\E \abs{ \hat \mu^n_h - \mu }^2 \le 2 \E \abs{ \hat \mu^n_h - \bar \mu^n_h }^2 + 2 \E \abs{ \bar \mu^n_h - \mu }^2 \,,
\end{equation}
where 
\begin{equation*}
	\bar \mu^n_h (t_0,u_0,x_0) = \frac{1}{n} \sum_{i=1}^n J_{h_2} (u_0 - \frac{i}{n}) K_{h_3} (x_0 - X_{\frac{i}{n}}(t_0)) \,.
\end{equation*}
The first part is controlled by the convergence of~\eqref{e:finite-pop-SDE} to~\eqref{e:graphon-SDE} (see~\cite{BayraktarChakrabortyWu2023}).
For the second part, we follow the idea of~\cite{MaestraHoffmann2022} and produce a Berstein concentration inequality.
The use of Bernstein's inequality here avoids the extra constants that arise from the change of measures in~\cite{MaestraHoffmann2022}, thanks to the independence of particles in the graphon mean-field system. 
It is worth noticing that all the constants appearing in the inequalities are global (independent of the plug-in arguments $t_0,u_0,x_0$), and we keep some of the explicit summations in the upper bound on purpose (as can be seen in Lemma~\ref{l:JK-est-mu} and~\ref{l:HJK-est-pi}).
Those properties preserve the integrability of the whole sums and maintain nice asymptotic behavior of the estimator $\hat G^n_{\vartheta}$.

In Section~\ref{s:minimax}, we perform a minimax analysis on the plug-in estimator of particle density $\mu$ and the drift coefficient $\beta$. 
We restrict our view to those particle systems with locally Hölder continuous density functions, the existence of which can be seen in several classical texts in Fokker-Planck equations such as~\cite{BogachevKrylovRocknerShaposhnikov2015}.
We present an alternative analysis on the pointwise behaviors of $\hat \mu^n_h$ and $\hat \pi^n_h$ with a change-of-measure strategy adapted from~\cite{MaestraHoffmann2022}.
This improves the pointwise errors obtained in Section~\ref{s:estimates} with the sacrifice of a constant factor depending on the value of $\mu$ near $(t_0,u_0,x_0)$. 
Balancing among the several error terms leads to optimal asymptotic upper bounds.
On the other hand, we find the (theoretical) lower bounds of the estimation error and compare them with the upper bounds, which demonstrates the optimality of our estimators. 
The proofs are given in Section~\ref{s:proof-minimax}.

\section{Model and estimators}\label{s:model}

\subsection{Setting, notation and assumptions}\label{s:setting}

Let us fix a finite time horizon $T > 0$. 
All observations are made within the time interval $[0, T]$. 
We denote by $\c_d$ the space of $\R^d$-valued continuous functions on $[0, T]$, i.e.\ $\c_d \defeq C([0, T]; \R^d)$. 
In more general cases, we write $C^k(\mathcal{X}; \mathcal{Y})$ for the space of $k$-times continuously differentiable functions defined on $\mathcal{X}$ taking values in $\mathcal{Y}$. 
Similarly, we write $L^p(\mathcal{X}; \mathcal{Y})$ for the space of $p$-th Lebesgue-integrable functions, and $W^{s,p}(\mathcal{X}; \mathcal{Y})$ for the Sobolev space. 
The position of $\mathcal{Y}$ is often neglected if $\mathcal{Y} = \R$. 

An $\R^d$-valued function $f$ can be written componentwise $(f_k)_{1 \le k \le d}$. 
The Fourier transform of a function $f: \R^d \to \R^d$ is defined componentwise via 
\begin{equation*}
	\f_{\R^d} f (\xi) = (\f_{\R^d} f_k (\xi))_{1 \le k \le d} = \left( \int_{\R^d} e^{-i x \xi} f_k(x) dx \right)_{1 \le k \le d} \,.
\end{equation*}
This will be applied to the drift coefficients in our deconvolution method.

We usually have the following assumptions on the graphon mean field system~\eqref{e:graphon-SDE}.
\begin{condition}
	\label{cd:coeffs}
	\begin{enumerate}
		\item The drift coefficient $b: \R^d \times \R^d \to \R^d$ is bounded and has bounded first derivatives. 
		It is Lipschitz continuous in the sense that there exists some constant $C > 0$ such that
		\begin{equation}
			\label{e:b-lip}
			\abs{b(x,y) - b(x',y')} \le C(\abs{x-x'} + \abs{y-y'}) \,, \qquad x, x', y, y' \in \R^d \,.
		\end{equation}
		\item The drift coefficient $b$ takes the form $b(x, y) = F(x - y) + V(x)$, where $F, V \in W^{1,p} (\R^d)$ for $p = 1,2,\infty$.
		
		\item The diffusion coefficient $\sigma: \R^d \to \R^{d \times d}$ is Lipschitz in the operator norm on $\R^{d \times d}$, i.e.\ there exists some constant $C > 0$ such that
		\begin{equation*}
			\norm{\sigma(x) - \sigma(x')} \le C \abs{x - x'} \,, \qquad x, x' \in \R^d \,,
		\end{equation*}
		where $\norm{\cdot}$ is the operator norm of ${d \times d}$ matrices. 
		\item The diffusion coefficient $\sigma$ is bounded in the sense that there exist constants $\sigma_\pm > 0$ such that 
		\begin{equation*}
			\sigma_{-}^2 I \preceq \sigma \sigma^T \preceq \sigma_{+}^2 I \,,
		\end{equation*}
		where two square matrices $M$ and $N$ satisfy $M \preceq N$ if $N - M$ is positive semi-definite. 
	\end{enumerate}
\end{condition}

Recall that the types of particles are indexed by $I = [0, 1]$, 
and the interaction strength between particles of two types is given by a graphon function $G: I \times I \to [0, 1]$. 
We consider the following conditions for the structure of the graphon function. 
\begin{condition}
	\label{cd:graphon}
	\begin{enumerate}
		\item The graphon function is piecewise Lipschitz in the sense that there exists a constant $C > 0$ and a finite partition $\bigcup_{j \in J} I_j$ of $I$, such that 
		\begin{equation*}
			\abs{G(u,v) - G(u',v')} \le C(\abs{u-u'} + \abs{v-v'}) \,, \qquad (u,v), (u',v') \in I_i \times I_j, \; i,j \in J \,.
		\end{equation*}

		\item The graphon function $G$ has the form $G(u,v) = g(u-v)$, where $g: \R \to [0, 1]$ is a Lipschitz continuous function with $g(0) = g_0 \in (0, 1]$ a given constant. 

		\item Upon item (2), we have further that the Fourier transform of $g$ is in $L^1 \cap L^2$ and decays fast enough, so that 
		\begin{equation*}
			\tilde{r}^2 \int_{\abs{w}>\tilde r} \abs{\f g(w)}^2 dw \to 0
		\end{equation*}
		as $\tilde r \to \infty$. 
	\end{enumerate} 
\end{condition}

Finally, we examine the initial state of the system. 
\begin{condition}
	\label{cd:init-data}
	We denote by $\prob(S)$ the space of probability measures on a Polish space $S$ (e.g.\ $\R^d$, $\c_d$). 
	\begin{enumerate}
		\item The initial distributions $\mu_{0, u}(dx)$ admit density functions $x \mapsto \mu(0, u, x)$ with respect to the Lebesgue measure on $\R^d$. 
		There exist some constant $c_0 > 0$, $c_1 \ge 1$ such that 
		\begin{equation*}
			\sup_{u \in I} \int_{\R^d} \exp (c_0 \abs{x}^2) \mu(0,u,x) \le c_1 \,.
		\end{equation*}

		\item There exists a constant $C > 0$ and a finite collection of intervals $\{I_j\}_{j \in J}$ such that $\bigcup_{j \in J} I_j = I$, and 
		\begin{equation*}
			\was_2(\mu_{0,u}, \mu_{0,v}) \le C\abs{u-v} \,, \qquad u, v \in I_j, \quad j \in J \,,
		\end{equation*}
		where $\was_2: \prob(\R^d) \times \prob(\R^d) \to [0, \infty]$ is the Wasserstein 2-distance. 

		\item There exists a function $\rho_I \in L^2 \cap L^\infty (\R^d)$ such that $\abs{\mu_{0,u} - \mu_{0,v}} \le \rho_I \abs{u-v}$ almost everywhere, for every $u,v \in I$. 
	\end{enumerate}
\end{condition}

The (continuously indexed) graphon mean-field system built on appropriately chosen conditions from above has dynamics 
\begin{align*}
	d X_u(t) = \int_I \int_{\R^d} b(X_u(t),x) G(u,v) \mu_{t,v}(dx) dv dt + \sigma(X_u(t)) dB_u(t) \,, \qquad X_u(0) \sim \mu_{0,u} \,,
\end{align*}
for $u \in I$, $t \in [0, T]$. 
Define the drift term 
\begin{equation*}
	\beta(t,u,x,\mu_t) \defeq \int_I \int_{\R^d} b(x,y) G(u,v) \mu_{t,v}(dy) dv 
\end{equation*}
and observe that $\beta: [0, T] \times I \times \R^d \times \prob(\R^d)^I$ is measurable. 
We will always abbreviate it as $\beta(t,u,x)$ (ignoring the mean-field argument). 
Under Condition~\ref{cd:coeffs}(1)(3), we know that $\beta(t,u,\cdot)$ is Lipschitz continuous and has at most linear growth for every $t \in [0, T]$ and $u \in I$. 
This means $\mu_{t,u}$ is the unique weak solution to the associated Fokker-Planck equation of the diffusion process $dX_u(t) = \beta(t,u, X_u(t)) dt + \sigma(X_u(t)) dB_u(t)$,
and the map $I \ni u \mapsto (\mu_{t,u})_{t \in [0, T]} \in \prob(\c_d)$ is measurable due to Proposition 2.1 in~\cite{BayraktarChakrabortyWu2023}.
Further, with Condition~\ref{cd:coeffs}(4) and~\ref{cd:init-data}(1), every $\mu_{\cdot,u}$ admits a density function $\mu(t,u,x)$ with respect to the Lebesgue measure on $[0, T] \times \R^d$ (see~\cite{BogachevKrylovRocknerShaposhnikov2015}).
Note that $\mu: [0, T] \times I \times \R^d \to \R_+$ is measurable. 
We claim that the densities are asymptotically bounded.
\begin{proposition}
	\label{pp:mu-L2-bound}
	Assume Conditions~\ref{cd:coeffs}(1)(3)(4), \ref{cd:init-data}(1), and that $b$ is almost everywhere bounded. 
	There exists some $C, R > 0$ such that, for every $p > d+2$ and every bounded open set $U$ disjoint from the closed ball $\overline{B(0,R)}$, we have for all $t \in (0,T)$ and $u \in I$ that 
	\begin{equation*}
		\norm{\mu_{t,u}}_{L^\infty(U)} \le C \norm{\mu_{0,u}}_{L^\infty(U)} + C t^{\frac{p-d-2}{2}} (1 + \norm{b}_\infty^p) \,.
	\end{equation*} 
	As a consequence, 
	\begin{equation*}
		\sup_{t \in [0, T], u \in I} \norm{\mu_{t,u} \one{ \{\abs{x} > R\} } }_2 < \infty \,.
	\end{equation*} 
\end{proposition}
The proof will be given in Appendix~\ref{s:pf-technical}. 
It also shows the $L^2$-integrability of the density function $\mu_{t,u}$ at any $t \in (0,T)$ (see also Corollary~8.2.2, \cite{BogachevKrylovRocknerShaposhnikov2015}). 

Our goal in the next section is to construct estimations of the functions $\mu(t,u,x)$ and $\beta(t,u,x)$, to give an estimation of $G(u,v) = g(u-v)$. 
We will use the $L^2$-distance in the probability space to describe our estimation errors.

\subsection{Plug-in estimators}\label{s:plugin}

To estimate the underlying functions described in the previous paragraph, we make continuous-time observations of the $n$-particle system~\eqref{e:finite-pop-SDE},
\begin{align*}
	X^n_i(t) & = X_{\frac{i}{n}}(0) + \int_0^t \frac{1}{n} \sum_{j=1}^n b(X^n_i(s), X^n_j(s)) g^n_{ij} ds + \int_0^t \sigma(X^n_i(s)) dB_{\frac{i}{n}}(s) \,,\\ 
	& \qquad i = 1, \dots, n \,,
\end{align*}
where $g^n_{ij} = G(\frac{i}{n}, \frac{j}{n})$ for every $i, j \in \{1, \dots, n\}$. 
This finite system shows some consistency with respect to the mean-field system in the following sense.
\begin{lemma}[Theorem 3.2, \cite{BayraktarChakrabortyWu2023}]
	\label{l:conv-of-system}
	Assume Conditions~\ref{cd:coeffs}(1)(3), \ref{cd:graphon}(1), and~\ref{cd:init-data}(1)(2) hold. 
	Then there exists a constant $C > 0$ such that 
	\begin{equation}
		\sup_{t \in [0, T]} \max_{1 \le i \le n} \E \abs{X^n_{i}(t) - X_{\frac{i}{n}}(t)}^2 \le \frac{C}{n} \,.
	\end{equation}
\end{lemma}

\subsubsection*{Kernel Interpolation}
We introduce an \emph{HJK-kernel} adapted from~\cite{MaestraHoffmann2022}. 
Choose three functions $H \in C^1_c(\R)$, $J \in C^1_c(\R)$, $K \in C^1_c(\R^d)$ that are non-negative and normalized:
\begin{equation*}
	\int_\R H(t) dt = \int_{\R} J(u) du = \int_{\R^d} K(x) dx = 1 \,,
\end{equation*}
and have order (at least) 1:
\begin{equation*}
	\int_\R t H(t) dt = \int_\R u J(u) du = \int_{\R^d} x_i K(x) dx = 0 \,, \quad i = 1, \dots, d \,.
\end{equation*}
With the bandwidth vector $h = (h_1, h_2, h_3) \in \R_+^3$, the dilations are defined by 
\begin{equation*}
	H_{h_1}(t) = h_1^{-1} H(h_1^{-1} t) \,, \quad J_{h_2}(u) = h_2^{-1} J(h_2^{-1} u) \,, \quad K_{h_3}(x) = h_3^{-d} K(h_3^{-1} x) \,,
\end{equation*}
and the products are written as
\begin{equation*}
	(J \otimes K)_h (u,x) = J_{h_2}(u) K_{h_3}(x) \,, \qquad (H \otimes J \otimes K)_h(t,u,x) = H_{h_1}(t) J_{h_2}(u) K_{h_3}(x) \,.
\end{equation*}
Due to the use of floating bandwidth, we choose without loss of generality kernels $H, J, K$ supported in the closed unit ball (in the space where they are defined). 

With a given number $n$ of particles, we run the finite system $(X^n_i)_{i=1,\dots,n}$ over the time interval $[0, T]$. 
This gives us the empirical distribution 
\begin{equation*}
	\mu^n_t (du, dx) = \frac{1}{n} \sum_{i=1}^{n} \delta_{X^n_i(t)} (dx) \delta_{\frac{i}{n}} (du) \,, \qquad t \in [0, T] \,.
\end{equation*}
Using the JK part of the kernel to interpolate it gives a plug-in estimator of the density $\mu$: 
\begin{equation}
	\label{e:JK-density}
	\hat \mu^n_h (t, u, x) \defeq
		\JK_h \ast \mu^n_{t}(u, x) = 
		\frac{1}{n} \sum_{i=1}^{n} J_{h_2}(u - \frac{i}{n}) K_{h_3} (x - X^n_i(t))  
\end{equation}
for $t \in [0, T]$, $u \in I$, $x \in \R^d$. 

We also look at an auxiliary quantity $\pi: [0, T] \times I \times \R^d \to \R^d$ defined by
\begin{equation*}
	\pi(t, u, x) \defeq \beta(t, u, x) \mu(t, u, x) \,.
\end{equation*}
A discrete approximation is given by
\begin{equation*}
	\pi^n(dt, du, dx) \defeq \frac{1}{n} \sum_{i=1}^{n} \delta_{X^n_i(t)} (dx) \delta_{\frac{i}{n}} (du) d X^n_i(t) \,,
\end{equation*}
so for any test function $f$, we have 
\begin{equation*}
	\int_{[0, T] \times I \times \R^d} f(t, u, x) \pi^n(dt, du, dx) = \int_0^T \frac{1}{n} \sum_{i=1}^{n} f(t, \frac{i}{n}, X^n_i(t)) dX^n_i(t) 
\end{equation*}
as a stochastic integral. 
Using the HJK kernel to interpolate it gives a plug-in estimator of $\pi$:
\begin{align}
	\label{e:HJK-pi}
	\hat \pi^n_h (t, u, x) & \defeq \HJK_h \ast \pi^n(t, u, x) \\
	\nonumber
	& = \int_0^T \frac{1}{n} \sum_{i=1}^{n} H_{h_1}(t - s) J_{h_2} (u - \frac{i}{n}) K_{h_3} (x - X^n_i(s)) d X^n_i(s)  
\end{align}
for $t \in [0, T]$, $u \in I$, $x \in \R^d$. 
That leads to a plug-in estimator of $\beta$, 
\begin{equation}
	\label{e:HJK-beta}
	\hat \beta_{h,\kappa} \defeq \frac{\hat \pi^n_h}{\hat \mu^n_h \lor \kappa_2} \,,
\end{equation}
where $\kappa_2 > 0$ is a cut-off parameter to prevent the denominator from getting too large.

\subsubsection*{Deconvolution}

The deconvolution method is usually applied to obtain a function $f$ from the convolution $f \ast g$. 
We follow the ideas of~\cite{Johannes2009} and~\cite{MaestraHoffmann2022}. 
Here, we only present the definitions and estimators, while delaying the full intuitions to Appendix~\ref{s:intu-est}. 

To apply the Fourier transform on the index space $I = [0, 1]$, we consider the \emph{zero extension} to all measurable functions defined in $I$. 
With some abuse of notation, we let 
\begin{equation*}
	\mu(t,u,x) = \begin{cases}
		\mu(t,u,x) \,, & u \in I \\
		0 \,, & u \in \R \setminus I 
	\end{cases} \,, \qquad 
	\hat \mu^n_h (t,u,x) = \begin{cases}
		\hat \mu^n_h (t,u,x) \,, & u \in I \\
		0 \,, & u \in \R \setminus I 
	\end{cases} \,,
\end{equation*}
\begin{equation*}
	\pi (t,u,x) = \begin{cases}
		\pi (t,u,x) \,, & u \in I \\
		0 \,, & u \in \R \setminus I 
	\end{cases} \,, \qquad 
	\hat \pi^n_h (t,u,x) = \begin{cases}
		\hat \pi^n_h (t,u,x) \,, & u \in I \\
		0 \,, & u \in \R \setminus I 
	\end{cases} \,,
\end{equation*}
\begin{equation*}
	\beta (t,u,x) = \begin{cases}
		\beta (t,u,x) \,, & u \in I \\
		0 \,, & u \in \R \setminus I 
	\end{cases} \,, \qquad 
	\hat \beta^n_{h,\kappa} (t,u,x) = \begin{cases}
		\hat \beta^n_{h,\kappa} (t,u,x) \,, & u \in I \\
		0 \,, & u \in \R \setminus I 
	\end{cases} \,.
\end{equation*}
Then we define the Fourier transform of function $f$ supported on $I$ via 
\begin{equation*}
	\f_I f (w) = \int_I e^{-iwu}f(u) du = \int_\R e^{-iwu}f(u) du \,.
\end{equation*}
Note that we may view $\f_I$ as a linear operator on function-valued functions, and it admits an inverse transform on $L^2$-spaces.

In addition, we consider a linear operator $\lin_\phi$ on time-dependent functions, defined by 
\begin{equation*}
	\lin_\phi f = \int_0^T f(t) \phi(t) dt \,,
\end{equation*}
where $\phi \in L^\infty([0,T];\C)$ has compact support in $(0,T)$, such that $\int_0^T \phi(t) dt = 0$ (we denote this subspace of $L^\infty$ functions by $\dot{L}^\infty$). 
We write $\lin$ for $\lin_\phi$ when $\phi$ is fixed and the context has no ambiguity. 
The intuition of this operator is also explained in Appendix~\ref{s:intu-est}. 

\subsubsection*{Main estimator and its convergence}

Finally, with some additional cutoff parameters, we introduce our estimator of the graphon function
\begin{equation}
	\label{e:estGdef}
	\hat G^n_{\vartheta} (u_0, v_0) \defeq g_0 \cdot \frac{\norm{\f_{I}^{-1} \paren[\big]{\frac{\t \hat \beta^n_{h, \kappa, r}}{\t \hat \mu^n_{h,r}} \one{\{\abs{\t \hat \mu^n_{h,r}} > \kappa_1, \abs{w} \le \tilde r\}}} (u_0 - v_0)}_{L^2(\R^d)}}{\norm{\f_{I}^{-1} \paren[\big]{\frac{\t \hat \beta^n_{h, \kappa, r}}{\t \hat \mu^n_{h,r}} \one{\{\abs{\t \hat \mu^n_{h,r}} > \kappa_1, \abs{w} \le \tilde r\}}} (0)}_{L^2(\R^d)} \lor \kappa_0} \,,
\end{equation}
where $\t \defeq \f_I \f_{\R^d} \mathcal{L}_\phi$, and
\begin{equation*}
	\hat \mu^n_{h,r} \defeq \hat \mu^n_h \one{\{\abs{x} \le r\}} \,, \qquad \hat \beta^n_{h, \kappa, r} \defeq \hat \beta^n_{h,\kappa} \one{\{\abs{x} \le r\}}  \,.
\end{equation*}
We will explain the intuition of this estimator in Appendix~\ref{s:intu-est}.

For the estimate to converge, we need a further (relatively strong) assumption on some data intrinsic to the particle system.
\begin{assumption}
	\label{as:FLmu}
	Given Conditions~\ref{cd:coeffs}, \ref{cd:graphon}, and \ref{cd:init-data}, there exists $\phi \in \dot{L}^\infty([0, T]; \C)$, compactly supported in $(0,T)$, such that $\f_I \f_{\R^d} \lin_\phi \mu \neq 0$ almost everywhere on $\R \times \R^d$. 
\end{assumption}

\begin{theorem}[Main theorem]
	\label{t:main}
	Assume Conditions~\ref{cd:coeffs}, \ref{cd:graphon}, \ref{cd:init-data}, and Assumption~\ref{as:FLmu}. 
	Take $\phi$ given by Assumption~\ref{as:FLmu}.
	There exists a function $\upbnd$ in the number of particles $n$ and the parameters 
	\begin{equation*}
		\vartheta = (h_1, h_2, h_3, \kappa_0, \kappa_1, \kappa_2, r, \tilde r)
	\end{equation*}
	such that 
	\begin{equation}
		\max_{u_0, v_0 \in I} \E \abs{ \hat G^n_\vartheta (u_0, v_0) - G(u_0, v_0) }^2 \le \upbnd(n, \vartheta) \,,
	\end{equation}
	whenever $n, r, \tilde r$ are sufficiently large, $h_1, h_2, h_3, \kappa_1 > 0$ are sufficiently small, and $\kappa_0, \kappa_1, \kappa_2 > 0$ such that $\kappa_0 < g_0 \norm{F}_2$ and
	\begin{equation*}
		\kappa_2 < \inf\{\mu(t,u,x) \mid t \in \supp(\phi), u \in I, \abs{x} \le r\} \,.
	\end{equation*}  

	Moreover, there exists a sequence of choices of the parameters $(\vartheta_n)_{n \in \N}$ such that $\upbnd(n, \vartheta_n) \to 0$ as $n \to \infty$. 
\end{theorem}

\begin{remark}
	Notice that the bound $\upbnd(n, \vartheta)$ is uniform over all $u_0, v_0 \in I$. 
	Then the \emph{cut distance} from an estimator $\hat G$ to the true graphon function $G$ is bounded by 
	\begin{align*}
		d_\square (\hat G, G) & \defeq \sup_{J_1, J_2 \in \b(I)} \abs{ \int_{J_1 \times J_2} \hat G(u,v) - G(u,v) dudv } \\
		& \le \int_{I \times I} \abs{\hat G(u,v) - G(u,v)} dudv \,.
	\end{align*}
	This implies 
	\begin{equation*}
		\E \paren[\big]{ d_\square (\hat G^n_\vartheta, G) } \le \int_{I \times I} \E \abs{ \hat G^n_\vartheta (u, v) - G(u, v) } dudv \le \upbnd(n, \vartheta)^{1/2} \,.
	\end{equation*}
	Hence we also have the convergence of the estimator $\hat G^n_\vartheta$ in the cut metric.  
\end{remark}

\begin{remark}
	The Assumption~\ref{as:FLmu} is made for the purpose of dominated convergence, and this is standard in a variety of applications of the deconvolution method~\cite{Johannes2009}. 
	Yet it is nontrivial to verify, as it involves the distribution through the whole time interval.
	We discuss this further in Appendix~\ref{s:reduc-FLmu}. 
\end{remark}

\section{Convergence of estimators}\label{s:main}

\subsection{Error bounds for density and drift}\label{s:estimates}

We give estimates of the particle density $\mu$ and the intermediate quantity $\pi$ in the general setting. 
These ultimately contribute to the estimate of the graphon function. 
From this section onwards, we will keep using the asymptotic comparison symbol $\lesssim$, where $f \lesssim g$ means there exists some constant $c > 0$ such that $f \le cg$.
In addition, we write $f \lesssim_q g$ if $f \le cg$ for some constant $c$ depending on the quantity $q$ (e.g.\ time horizon $T$, dimension $d$). 

\subsubsection*{Estimates of particle density $\mu(t,u,x)$}
\begin{lemma}
	\label{l:JK-est-mu}
	Assume that Conditions~\ref{cd:coeffs}(1)(3), \ref{cd:graphon}(1), and~\ref{cd:init-data}(1)(2) hold. 
	For $t_0 \in (0,T)$, $u_0 \in (0,1)$, $x_0 \in \Omega$, we have
	\begin{align*}
		& \E \abs{\hat \mu^n_h (t_0, u_0, x_0) - \mu(t_0, u_0, x_0)}^2 \le \\
		& \qquad C ( n^{-2} h_3^{-2-2d} \norm{\grad K}_\infty^2 \sum_{i=1}^n J_{h_2}(u_0 - \frac{i}{n})^2 \\
		& \qquad + n^{-2} \sum_{i=1}^{n} J_{h_2} (u_0 - \frac{i}{n})^2 \norm{K_{h_3} (x_0 - \cdot)}^2_{L^2(\mu_{t_0, \frac{i}{n}})} + n^{-2} h_2^{-2} h_3^{-2d} \norm{J}_\infty^2 \norm{K}_\infty^2  \\
		& \qquad + n^{-2} h_3^{-2-2d} \norm{J}_2^2 \norm{\grad K}_\infty^2 + n^{-3} h_2^{-2} \norm{\grad J}_\infty^2 \sum_{i=1}^{n} \norm{K_{h_3}(x_0 - \cdot)}_{L^2(\mu_{t_0, \frac{i}{n}})}^2 )  \\
		& \qquad + \abs{\JK_h \ast \mu_{t_0} (u_0, x_0) - \mu(t_0, u_0, x_0)}^2  \,,
	\end{align*}
	where 
	\begin{equation*}
		\JK_h \ast \mu_{t_0} (u_0, x_0) = \int_I \int_{\R^d} \JK_h(u_0-u, x_0-x) \mu(t_0,u,x) dxdu \,.
	\end{equation*}
\end{lemma}

Integrating the above pointwise errors, we have the following $L^2$-error on the estimator $\mu^n_h$. 
\begin{corollary}
	\label{c:err-bound-mu-overall}
	Assume the same hypothesis as in Lemma~\ref{l:JK-est-mu}.
	Fix a compact interval $[\tau_1, \tau_2] \subset (0,T)$.
	For $n \gg 1$, $h_2, h_3 \ll 1$, $r \gg 1$, we have 
	\begin{align*}
		& \int_{\tau_1}^{\tau_2} \int_I \int_{\R^d} \E \abs{\hat \mu^n_{h,r}(t, u, x) - \mu(t, u, x)}^2 dx du dt \le C ( \theta_{2,\mu}(r) +  \theta_{3,\mu}(h) + \\
		& \qquad r^d (n^{-2} h_3^{-2-2d} + n^{-2} h_2^{-2} h_3^{-2d}) + n^{-1} h_2^{-1}  h_3^{-d} + n^{-2} h_2^{-2} h_3^{-d} ) \,,
	\end{align*}
	where $C$ is a constant depending on ${T,d,b,J,K}$. 
        Here $\theta_{2,\mu}: \R_+ \to \R_+$ is a function such that $\theta_{2,\mu}(r) \to 0$ as $r \to \infty$, and
	$\theta_{3,\mu}: \R_+^3 \to \R_+$ is a function such that $\theta_{3,\mu}(h) \to 0$ as $h_2 + h_3 \to 0$. 
\end{corollary}

\subsubsection*{Estimates of the drift term $\beta(t,u,x)$}
\begin{lemma}
	\label{l:HJK-est-pi}
	Assume Conditions~\ref{cd:coeffs}(1)(3)(4), \ref{cd:graphon}(1), \ref{cd:init-data}(1)(2), and that $b$ is bounded. 
	Then, for $t_0 \in (0,T)$, $u_0 \in (0,1)$, $x_0 \in \R^d$, we have 
	\begin{align*}
		& \E \abs{\pi^n_h (t_0, u_0, x_0) - \pi (t_0, u_0, x_0)}^2 \le C ( T d^2 \sigma_+^2 n^{-1} h_1^{-2} h_2^{-2} h_3^{-2d} \\
		& \quad + T n^{-1} h_1^{-1} h_2^{-2} h_3^{-2-2d} \norm{b}_\infty^2 \norm{H}_2^2 \norm{J}_\infty^2 \norm{\grad K}_\infty^2  \\
		& \qquad \quad + T n^{-2} h_1^{-1} h_3^{-2d} \sum_{i=1}^n J_{h_2}(u_0-\frac{i}{n})^2 \\
		& \quad + n^{-2} h_1^{-1} h_2^{-2} h_3^{-2d} T \norm{b}_\infty^2 \norm{H}_2^2 \norm{J}_\infty^2 \norm{K}_\infty^2 \\
		& \qquad \quad + T n^{-2} \norm{b}_\infty^2 \int_0^T H_{h_1}^2(t_0 - t) \sum_{i=1}^{n} J_{h_2}^2(u_0 - \frac{i}{n}) \norm{K_{h_3}(x_0 - \cdot)}_{L^2(\mu_{t, \frac{i}{n}})}^2 dt \\
		& \quad + T n^{-2} h_1^{-1} h_2^{-1} h_3^{-2-2d} \norm{b}_\infty^2 \norm{H}_2^2 \norm{J}_2^2 \norm{\grad K}_\infty^2 \\
		& \qquad \quad + T n^{-2} h_1^{-1} h_2^{-1} h_3^{-2d} \norm{H}_2^2 \norm{J}_2^2 \norm{K}_\infty^2 \\
		& \qquad \quad + T n^{-3} h_2^{-4} \norm{b}_\infty^2 \norm{\grad J}_\infty^2 \int_0^T H_{h_1}^2(t_0 - t) \sum_{i=1}^{n} \norm{K_{h_3}(x_0 - \cdot)}_{L^2(\mu_{t, \frac{i}{n}})}^2 dt ) \\
		& \quad + \abs{\HJK_h \ast \pi (t_0, u_0, x_0) - \pi(t_0, u_0, x_0)}^2  \,.
	\end{align*} 
\end{lemma}

Recall that $\beta^n_{h,\kappa} = \frac{\pi^n_h}{\mu^n_h \lor \kappa_2}$. 
Integrating the above pointwise errors and using Corollary~\ref{c:err-bound-mu-overall}, we have the following $L^2$-error on the estimator $\beta^n_{h,\kappa}$. 
\begin{corollary}
	\label{c:err-bound-beta-overall}
	Assume the same hypothesis as in Lemma~\ref{l:HJK-est-pi}.
	Fix a compact interval $[\tau_1, \tau_2] \subset (0,T)$.
	For $n \gg 1$, $h_1, h_2, h_3 \ll 1$, $r \gg 1$, and $0 < \kappa_2 < \inf \{ \mu(t,u,x) \st t \in [\tau_1,\tau_2], u \in I, \abs{x} \le r \}$, we have 
	\begin{align*}
		& \int_{\tau_1}^{\tau_2} \int_{I} \int_{\R^d} \E \abs{\hat \beta^n_{h,\kappa,r} (t, u, x) - \beta (t, u, x)}^2 dx du dt \le \\
		& \qquad  C ( \kappa_2^{-2} \paren[\big]{ n^{-1} h_1^{-1} h_2^{-1} h_3^{-d} + n^{-2} h_1^{-1} h_2^{-4} h_3^{-d} } \\
		& \quad + \kappa_2^{-2} r^{d} (n^{-1} h_1^{-1} h_2^{-2} h_3^{-2-2d} +  n^{-1} h_1^{-2} h_2^{-2} h_3^{-2d}) \\
		& \quad + \kappa_2^{-2} (\theta_{3,\mu}(h) + \theta_{3,\pi}(h)) + \theta_{2, \beta}(r) ) \,,
	\end{align*}
    where the constant $C$ depends on ${T,d,b,\sigma,H,J,K}$.
	Here $\theta_{2, \beta}: \R_+ \to \R_+$ is some function such that $\theta_{2, \beta}(r) \to 0$ as $r \to \infty$,
	and $\theta_{3, \pi}: \R_+^3 \to \R_+$ is some function such that $\theta_{3, \pi}(h) \to 0$ as $h_1 + h_2 + h_3 \to 0$. 
\end{corollary}

Although we postponed the proofs to Section~\ref{s:estimates-proof}, we now justify the main result (Theorem~\ref{t:main}).

\subsection{Proof of main theorem}

\begin{proof}[Proof of Theorem~\ref{t:main}]
	\step[]
	For simplicity, we usually abbreviate $\hat G^n_\vartheta$ as $\hat G$, and similarly for $\hat \mu$ and $\hat \beta$. 

	Fix arbitrary $u_0, v_0 \in I$. 
	Write $\hat G$ as 
	\begin{equation*}
		\hat G (u_0, v_0) = g_0 \cdot \frac{\hat A(u_0 - v_0)}{\hat A(0) \lor \kappa_0} 
	\end{equation*}
	where
	\begin{equation*}
		\hat A(u) \defeq \norm{\f_{I}^{-1} \paren[\big]{\frac{\t \hat \beta^n_{h, \kappa, r}}{\t \hat \mu^n_{h,r}} \one{\{\abs{\t \hat \mu^n_{h,r}} > \kappa_1, \abs{w} \le \tilde r\}}} (u)}_{L^2(\R^d)} \,.
	\end{equation*}
	Also let $A(u) = g(u) \norm{F}_2$. 
	When $\kappa_0 < A(0)$, we have 
	\begin{align*}
		\abs{\hat G (u_0, v_0) - G(u_0, v_0) }^2& = g_0^2 \abs{ \frac{A(u_0 - v_0) (\hat A(0) - A(0))}{A(0) (\hat A(0) \lor \kappa_0)} + \frac{A(u_0 - v_0) - \hat A(u_0 - v_0)}{\hat A(0) \lor \kappa_0} }^2 \\
		& \lesssim \frac{G(u_0, v_0)^2 \abs{\hat A(0) \lor \kappa_0  - A(0)}^2}{(\hat A(0) \lor \kappa_0)^2 } + \frac{\abs{A(u_0 - v_0) - \hat A(u_0 - v_0)}^2}{(\hat A(0) \lor \kappa_0)^2} \\
		& \lesssim \kappa_0^{-2} \left( \abs{\hat A(0) - A(0)}^2 + \abs{A(u_0 - v_0) - \hat A(u_0 - v_0)}^2 \right) \,.
	\end{align*}
	So it suffices to bound the expressions $\E \abs{\hat A(u) - A(u)}^2$, and we will do that in the following steps.  

	\step[]
	By Minkowski's inequality, we have for each $u \in \R$ that
	\begin{align*}
		& \abs{ \norm{\f_{I}^{-1} \paren[\big]{\frac{\t \hat \beta^n_{h, \kappa, r}}{\t \hat \mu^n_{h,r}} \one{\{\abs{\t \hat \mu^n_{h,r}} > \kappa_1, \abs{w} \le \tilde r\}}} (u)}_{L^2(\R^d)} - g(u) \norm{F}_{L^2(\R^d)} } \\
		& \le \norm{\f_{I}^{-1} \paren[\big]{\frac{\t \hat \beta^n_{h, \kappa, r}}{\t \hat \mu^n_{h,r}} \one{\{\abs{\t \hat \mu^n_{h,r}} > \kappa_1, \abs{w} \le \tilde r\}}} (u)  - g(u) \f_{\R^d} F}_{L^2(\R^d)} \\
		& = \norm{\f_{I}^{-1} \left( \paren[\big]{\frac{\t \hat \beta^n_{h, \kappa, r}}{\t \hat \mu^n_{h,r}} \one{\{\abs{\t \hat \mu^n_{h,r}} > \kappa_1, \abs{w} \le \tilde r\}}  - \f_I g \f_{\R^d} F   } \right) (u)     }_{L^2(\R^d)} \,.
	\end{align*}
	Then
	\begin{align*}
		\E \abs{\hat A(u) - A(u)}^2 & \lesssim \E  \norm{ \f_I^{-1} \paren[\big]{\frac{\t \hat \beta^n_{h, \kappa, r}}{\t \hat \mu^n_{h,r}} \one{\{\abs{\t \hat \mu^n_{h,r}} > \kappa_1, \abs{w} \le \tilde r\}}  - \f_I g \f_{\R^d} F   } (u)  }_{L^2(\R^d)}^2 \\
		& \lesssim \tilde r^2 \E \int_{\abs{w} \le \tilde r} \norm{ \paren[\big]{\frac{\t \hat \beta^n_{h, \kappa, r}}{\t \hat \mu^n_{h,r}} \one{\{\abs{\t \hat \mu^n_{h,r}} > \kappa_1\}}  - \f_I g \f_{\R^d} F   } (w)  }_{L^2(\R^d)}^2 dw \\
		& \qquad \qquad + \norm{ \f_{\R^d} F \int_{\abs{w}>\tilde r} e^{iuw}  \f_I g (w) dw }_{L^2(\R^d)}^2 \,.
	\end{align*}
	For the second term, we observe by Parseval's identity that 
	\begin{equation*}
		\norm{\f_I^{-1} (\f_I g \f_{\R^d} F) (u)}_{L^2(\R^d)} = \abs{g(u)} \norm{F}_2 \le \norm{F}_2 < \infty \,,
	\end{equation*}
	so that 
	\begin{equation*}
		\norm{ \f_{\R^d} F \int_{\abs{w}>\tilde r} e^{iuw}  \f_I g(w) dw }_{L^2(\R^d)}^2 \le \norm{F}_2^2 \left( \int_{\abs{w} > \tilde r} \abs{\f_I g (w)} dw \right)^2 \to 0
	\end{equation*}
	as $\tilde r \to \infty$ due to Condition~\ref{cd:graphon}(3), at some rate $\tilde \theta(\tilde r)$ independent of $u$. 

	\step[]
	Now we look at the first term under Assumption~\ref{as:FLmu},
	\begin{equation*}
		\tilde r^2 \E \int_{\R} \int_{\R^d} \abs{ \paren[\big]{\frac{\t \hat \beta^n_{h, \kappa, r}}{\t \hat \mu^n_{h,r}} \one{\{\abs{\t \hat \mu^n_{h,r}} > \kappa_1, \abs{w} \le \tilde r\}}  - \frac{\t \beta}{\t \mu}   } (w, \xi)  }^2 d \xi d w
	\end{equation*}
	Notice that $\t \beta = (\f_I g) (\f_{\R^d} F) (\t \mu)$.
	Following~\cite{MaestraHoffmann2022}, we split the integrand by 
	\begin{equation*}
		\frac{\t \hat \beta}{\t \hat \mu} - \f_I g \f_{\R^d} F = \frac{(\t \hat \beta - \t \beta) + (\f_I g)(\f_{\R^d} F)(\t \mu - \t \hat \mu)}{\t \hat \mu} \,,
	\end{equation*}
	when the division is well-defined. 
	Then 
	\begin{align*}
		& \abs{ \frac{\t \hat \beta}{\t \hat \mu} \one{\{\abs{\t \hat \mu} > \kappa_1, \abs{w} \le \tilde r\}} - \f_I g \f_{\R^d} F }^2  \\
		& \quad \lesssim \kappa_1^{-2} \abs{\t \hat \beta - \t \beta}^2 
		+ \kappa_1^{-2} \abs{\f g}^2 \abs{\f F}^2 \abs{\t \mu - \t \hat \mu}^2 
		+ \abs{\f_{I} g \f_{\R^d} F (\one{\{\abs{\t \hat \mu} \le \kappa_1\}} + \one{\{\abs{w} > \tilde r\}})}^2 \\
		& \quad =: \a_1 + \a_2 + \a_3 \,.
	\end{align*}
	
	For~$\a_1$, the Parseval's identity gives
	\begin{align*}
		\int_{\R} \int_{\R^d} \kappa_1^{-2} \abs{\t \hat \beta - \t \beta}^2 (w, \xi) d \xi d w & = \kappa_1^{-2} \int_{\R} \int_{\R^d} \abs{\lin_\phi \hat \beta (u, x) - \lin_\phi \beta (u, x)}^2 dx du \\
		& \le \kappa_1^{-2} \norm{\phi}_2^2 \int_{\tau_1}^{\tau_2} \int_{I} \int_{\R^d} \abs{\hat \beta (t, u, x) - \beta (t, u, x)}^2 dx du dt \,,
	\end{align*}
	where $\supp(\phi) \subset [\tau_1, \tau_2]$.
	From Corollary~\ref{c:err-bound-beta-overall} we get 
	\begin{align*}
		& \E \int_{\R} \int_{\R^d} \kappa_1^{-2} \abs{\t \hat \beta (w, \xi) - \t \beta (w, \xi)}^2 d \xi d w \lesssim_{T,d,b,\sigma,H,J,K,\phi} \\
		& \qquad \qquad \kappa_1^{-2} \kappa_2^{-2} \paren[\big]{ n^{-1} h_1^{-1} h_2^{-1} h_3^{-d} + n^{-2} h_1^{-1} h_2^{-4} h_3^{-d} } \\
		& \qquad \quad + \kappa_1^{-2} \kappa_2^{-2} r^{d} (n^{-1} h_1^{-1} h_2^{-2} h_3^{-2-2d} + n^{-1} h_1^{-2} h_2^{-2} h_3^{-2d}) \\
		& \qquad \quad + \kappa_1^{-2} \kappa_2^{-2} (\theta_{3,\mu}(h) + \theta_{3,\pi}(h)) + \kappa_1^{-2} \theta_{2, \beta}(r) \,.
	\end{align*}
	
	For~$\a_2$, similarly we have 
	\begin{align*}
		\int_{\R} \int_{\R^d} \abs{\t \hat \mu - \t \mu}^2 (w, \xi) d \xi d w \le  \norm{\phi}_2^2 \int_0^T \int_{\R} \int_{\R^d} \abs{\hat \mu (t, u, x) - \mu (t, u, x)}^2 dx du dt \,.
	\end{align*}
	Also, $\abs{\f g} \le \norm{g}_1 \le 2$, and $\abs{\f F} \le \norm{F}_1 < \infty$. 
	Along with Corollary~\ref{c:err-bound-mu-overall} we have that 
	\begin{align*}
		& \E \int_{\R} \int_{\R^d} \kappa_1^{-2} \abs{\f g}^2 \abs{\f F}^2 \abs{\t \mu - \t \hat \mu}^2 d \xi d w \lesssim_{T,d,b,\sigma,J,K,\phi} \\
		& \quad \qquad \kappa_1^{-2} (\theta_{2,\mu}(r) +  \theta_{3,\mu}(h))  \\
		& \quad \quad + \kappa_1^{-2} r^d (n^{-2} h_3^{-2-2d} + n^{-2} h_2^{-2} h_3^{-2d}) + \kappa_1^{-2} ( n^{-1} h_2^{-1}  h_3^{-d} + n^{-2} h_2^{-2} h_3^{-d} ) \,,
	\end{align*}
	
	For $\a_3$, we first observe that 
	\begin{align*}
		& \E \abs{ \f_{I} g(w) \f_{\R^d} F(\xi) \one{\{\abs{\t \hat \mu} \le \kappa_1\}}(w, \xi) }^2 \\
		& \le \abs{\f_{I} g(w) \f_{\R^d} F(\xi) }^2 \E \left[ \one{\{\abs{\t \hat \mu - \t \mu} \ge \kappa_1\}} + \one{\{\abs{\t \mu} \le 2 \kappa_1\}} \right] \\
		& \le \abs{\f_{I} g(w) \f_{\R^d} F(\xi) }^2 \left( \kappa_1^{-2} \E \abs{\t (\hat \mu - \mu) (w, \xi)}^2 + \one{\{\abs{\t \mu} \le 2 \kappa_1\}}(w, \xi) \right) \,.
	\end{align*}
	Integrating the first part gives 
	\begin{align*}
		& \int_{\R} \int_{\R^d} \abs{\f g(w) \f F(\xi) }^2 \kappa_1^{-2} \E \abs{\t (\hat \mu - \mu) (w, \xi)}^2 d\xi dw \\
		& \le \kappa_1^{-2} \norm{g}_1^2 \norm{F}_1^2 \int_{\R} \int_{\R^d} \E \abs{\t (\hat \mu - \mu) (w, \xi)}^2 d\xi dw \\
		& \lesssim_{g, b, \phi} \kappa_1^{-2} \int_0^T \int_\R \int_{\R^d} \E \abs{\hat \mu^n_{h,r}(t, u, x) - \mu(t, u, x)}^2 dx du dt \,,
	\end{align*}
	which can be bounded in the same way as $\a_2$ using Corollary~\ref{c:err-bound-mu-overall}. 
	
	Integrating for the second part, we get 
	\begin{equation}
		\int_{\R} \int_{\R^d}  \abs{\f_{I} g(w) \f_{\R^d} F(\xi) }^2 \one{\{\abs{\t \mu} \le 2 \kappa_1\}} (w, \xi) d\xi dw \,.
	\end{equation}
	Under Assumption~\ref{as:FLmu}, we apply dominated convergence to see that this quantity goes to 0 as $\kappa_1 \to 0$.

	In addition,
	\begin{align*}
		& \int_0^T \int_\R \int_{\R^d} \E \abs{ \f_{I} g(w) \f_{\R^d} F(\xi) \one{\{\abs{w} > \tilde r\}}(w, \xi) }^2 \\
		& \le \int_0^T \int_\R \int_{\R^d} \abs{\f_I g(w) \one{\{\abs{w} > \tilde r\}}}^2 \abs{\f_{\R^d} F(\xi)}^2 \\
		& = T \norm{F}_2^2 \int_{\{\abs{w} > \tilde r\}} \abs{\f_I g(w)}^2 \,.
	\end{align*}
	Condition~\ref{cd:graphon}(3) guarantees that it converges to 0 faster than $\tilde r^{-2}$ as $\tilde r \to \infty$. 
	We denote the total convergence rate of $\a_3$ by $\theta_1 (\tilde r, \kappa_1)$. 

	To summarize, we define 
	\begin{align*}
		\upbnd(n, \vartheta) & = C ( \kappa_0^{-2} \tilde r^2 \tilde \theta(\tilde r) \\
		& + \tilde r^2 \kappa_0^{-2} \theta_1(\tilde r, \kappa_1) \\ 
		& + \tilde r^2 \kappa_0^{-2} \kappa_1^{-2} (\theta_{2, \beta}(r) + \kappa_2^{-2} (\theta_{2, \mu}(r) + \theta_{3, \mu}(h) + \theta_{3, \pi}(h))) \\
		& + \tilde r^2 \kappa_0^{-2} \kappa_1^{-2} \kappa_2^{-2} (n^{-1} h_1^{-1} h_2^{-1} h_3^{-d} + n^{-2} h_1^{-1} h_2^{-4} h_3^{-d}) \\
		& + \tilde r^2 \kappa_0^{-2} \kappa_1^{-2} \kappa_2^{-2} r^d (n^{-1} h_1^{-1} h_2^{-2} h_3^{-2-2d} + n^{-1} h_1^{-2} h_2^{-2} h_3^{-2d}) ) \,.
	\end{align*}
	Here the constant $C$ depends only on $T,d,b,\sigma,H,J,K$, which are fixed for the model. 
	The upper bound is independent of $u_0, v_0$, so we obtain the uniform bound presented in the theorem. 

	We can fix $\kappa_0$ and $\phi$. 
	Let $\vartheta_n = (h_1^{(n)}, h_2^{(n)}, h_3^{(n)}, \kappa_0, \kappa_1^{(n)}, \kappa_2(r^{(n)}), r^{(n)}, \tilde r^{(n)})$,
	where $r^{(n)}, \tilde r^{(n)} \to \infty$ slowly enough as $n \to \infty$, 
	$\kappa_2(r) = \frac{1}{2} \inf \{\mu(t,u,x) \st t \in \supp(\phi), u \in I, \abs{x} \le r\}$,
	and $\kappa_1^{(n)}, h_1^{(n)}, h_2^{(n)}, h_3^{(n)} \to 0$ accordingly. 
	We may guarantee that the quantities $\theta_1$, $\theta_{2, \beta}$, $\theta_{2, \mu}$, $\theta_{3, \mu}$, $\theta_{3, \pi}$ all converge to 0. 
	Then $\upbnd(n, \vartheta_n) \to 0$ as $n \to \infty$, finishing the proof. 
\end{proof}

\section{Minimax analysis on plug-in estimators}\label{s:minimax}

In Section~\ref{s:estimates} we presented upper bounds for the estimation error $\E \abs{\hat \mu^n_h - \mu}^2$ and $\E \abs{\hat \pi^n_h - \pi}^2$.
Those are similar to those given in~\cite{BolleyGuillinVillani2007} and are not tight, though convergent. 
However, the estimators themselves are indeed optimal whenever the parameters $h_1$, $h_2$, and $h_3$ are properly chosen.
In this section, we conduct a minimax analysis to study both the upper bounds and lower bounds of the estimation errors to witness the optimality. 

We first look at an improved upper bound on the error of the plug-in estimator $\hat \mu^n_h$. 
\begin{lemma}
	\label{l:improve-est-mu}
	\begin{enumerate}
		\item Assume Conditions~\ref{cd:coeffs}(1)(3), \ref{cd:graphon}(1), and~\ref{cd:init-data} hold.
		For every $t_0 \in (0,T)$, $u_0 \in (0,1)$, $x_0 \in \R^d$, we have 
		\begin{align*}
			& \E \abs{ \hat \mu^n_h (t_0,u_0,x_0) - \mu (t_0,u_0,x_0) }^2 \le \\
			& \qquad C_0 (n^{-1} h_2^{-1} h_3^{-d} \norm{J}_\infty^2 \norm{K}_2^2 + n^{-2} h_2^{-2} h_3^{-2d} \norm{J}_\infty^2 \norm{K}_\infty^2 \\
			& \qquad + n^{-2} h_3^{-2-2d} \norm{J}_\infty^2 \norm{\grad K}_\infty^2 + n^{-2} h_2^{-2} h_3^{-d} \norm{\grad J}_\infty^2 \norm{K}_2^2 ) \\
			& \qquad + \abs{ \JK_h \ast \mu_{t_0} (u_0,x_0) - \mu(t_0,u_0,x_0) }^2  \,,
		\end{align*}
		where $C_0 > 0$ is independent of the bandwidths $h_2, h_3$ and the number of particles $n$. 
	
		\item Assume further that there exist some $p > 2$ and $c_p > 0$ such that 
		\begin{equation*}
			\was_{p} (\mu_{0,u}, \mu_{0,v}) \le c_p \abs{u-v},, \qquad \forall u,v \in I,.
		\end{equation*}
		Then, for every $t_0 \in (0,T)$, $u_0 \in (0,1)$, $x_0 \in \R^d$, we have 
		\begin{align*}
			& \E \abs{ \hat \mu^n_h (t_0,u_0,x_0) - \mu (t_0,u_0,x_0) }^2 \le \\
			& \qquad C_0 (n^{-1} h_2^{-1} h_3^{-d} \norm{J}_\infty^2 \norm{K}_2^2 + n^{-2} h_2^{-2} h_3^{-2d} \norm{J}_\infty^2 \norm{K}_\infty^2 \\
			& \qquad + n^{-2} h_3^{-2-\frac{p+2}{p}d} \norm{J}_\infty^2 \norm{\grad K}_\infty^2 + n^{-2} h_2^{-2} h_3^{-d} \norm{\grad J}_\infty^2 \norm{K}_2^2 ) \\
			& \qquad + \abs{ \JK_h \ast \mu_{t_0} (u_0,x_0) - \mu(t_0,u_0,x_0) }^2  \,,
		\end{align*}
		where $C_0 > 0$ is independent of the bandwidths $h_1, h_2, h_3$ and the number of particles $n$. 
	\end{enumerate}
\end{lemma}

There is also an improved upper bound on the error of the plug-in estimator $\hat \beta^n_{h,\kappa}$.
\begin{lemma}
	\label{l:improve-est-beta}
	Assume Conditions~\ref{cd:coeffs}(1)(3)(4), \ref{cd:graphon}(1), \ref{cd:init-data}.
	Assume further that $b$ is bounded and has bounded first and second derivatives. 
	For every $t_0 \in (0,T)$, $u_0 \in (0,1)$, $x_0 \in \R^d$, we have
	\begin{align*}
		& \E \abs{\hat \pi^n_h (t_0,u_0,x_0) - \pi (t_0,u_0,x_0)}^2 \lesssim \\
		& \qquad \quad  n^{-1} h_1^{-1} h_2^{-1} h_3^{-d} + n^{-2} h_1^{-1} h_2^{-2} h_3^{-1-2d} + n^{-2} h_1^{-1} h_3^{-2-2d} \\
		& \qquad \quad + \abs{\HJK_h \ast \pi (t_0,u_0,x_0) - \pi(t_0,u_0,x_0)}^2 \,.
	\end{align*}
\end{lemma}

\begin{remark}
	We will use the above results for the minimax analysis in the next step.  
	Yet, they are not ideal for estimating the total error of $L^2([0,T] \times I \times \R^d)$ (as used in the proof of Theorem~\ref{t:main}). 
	It relies on the local boundedness of the density function $\mu$, which follows from a crude estimate of the form $\mu(t,u,x) \le \exp (c (1 + \abs{x}^2))$ (see, for instance, Corollary 8.2.2, \cite{BogachevKrylovRocknerShaposhnikov2015}).
\end{remark}

\subsection{Anisotropic Hölder smoothness classes}

In the estimates of Lemma~\ref{l:improve-est-mu}, all items are explicitly quantitative except the bias term 
\begin{align*}
	& \JK_h \ast \mu_{t_0} (u_0,x_0) - \mu(t_0,u_0,x_0) =\\ 
	& \qquad \int_\R \int_{\R^d} J_{h_2}(u_0-u) K_{h_3}(x_0-x) (\mu(t_0,u,x) - \mu(t_0,u_0,x_0)) dxdu \,.
\end{align*}
The analysis of this quantity relies on some continuity of the density function $\mu$. 
Here, we introduce a specific class of particle systems following the idea in~\cite{MaestraHoffmann2022}. 

\begin{definition}
	Let $\alpha = (\alpha_1, \dots, \alpha_d) \in \N^d$ be a multi-index. 
	Its norm is given by 
	\begin{equation*}
		\abs{\alpha} = \sum_{i=1}^d \alpha_i \,,
	\end{equation*}
	and the differential operator of order $\alpha$ is defined by $D^\alpha = \p^{\alpha_1}_1 \cdots \p^{\alpha_d}_d$.
\end{definition}

\begin{definition}
	Let $U \subset \R^d$ be an open neighborhood of a point $x_0 \in \R^d$. 
	We say a function $f:\R^d \to \R$ belongs to the \emph{$s$-Hölder continuity class at $(x_0,U)$} with $s > 0$ if 
	for every $x, y \in U$ and every multi-index $\alpha$ with $\abs{\alpha} \le s$, we have
	\begin{equation*}
		\abs{D^\alpha f(x) - D^\alpha f(y)} \le C \abs{x-y}^{s - \floor{s}} \,,
	\end{equation*}
	where $C = C(f,U)$ is the smallest constant that satisfies the above inequality.  
	We denote this class of functions by $\h^s(x_0)$. 
	The $\h^s$-norm in this class is defined by 
	\begin{equation*}
		\norm{f}_{\h^s(x_0)} = \sup_{x \in U} \abs{f(x)} + C(f,U) \,.
	\end{equation*}
\end{definition}

\subsection{Minimax estimation for density}

Notice that the particle density function $\mu$ solves the following system of equations~\cite{Coppini2022note}
\begin{equation*}
	\p_t \mu_{t,u} = - \grad \cdot \left( \mu_{t,u} \int_I \int_{\R^d} b(\cdot,y) G(u,v) \mu_{t,v}(dy) dv  \right) + \frac{1}{2} \sum_{i,j=1}^d \p_{ij}^2 ((\sigma \sigma^T)_{ij} \mu_{t,u}) \,, \quad u \in I \,.
\end{equation*}
This is a system of fully coupled Fokker-Planck equations, and the solution is uniquely determined by $(b,\sigma,G,\mu_0) \in \mathcal{P}$, where $\mathcal{P}$ is the class of $(b,\sigma,G,\mu_0)$ satisfying Condition~\ref{cd:coeffs}(1)(3), \ref{cd:graphon}(1), and~\ref{cd:init-data}. 
We denote by $\mathcal{P} \ni (b,\sigma,G,\mu_0) \mapsto \mu = S(b,\sigma,G,\mu_0)$ the solution operator. 
We consider a specific class of coefficients and initial data.

For $s > 0$, we define 
\begin{equation*}
	\a^{s}(t_0,u_0,x_0) = \{ (b,\sigma,G,\mu_0) \in \mathcal{P} \st \mu = S(b,\sigma,G,\mu_0), \mu_{t_0,u_0} \in \h^{s}(x_0) \} \,,
\end{equation*}
and set
\begin{equation*}
	\a^{s}(t_0,x_0) = \bigcap_{u \in I} \a^{s}(t_0,u,x_0) \,.
\end{equation*}
Moreover, we consider a restriction of this class 
\begin{equation*}
	\a^{s}_L(t_0,x_0) = \{ (b,\sigma,G,\mu_0) \in \mathcal{P} \st \norm{ S(b,\sigma,G,\mu_0)_{t_0,u_0} }_{\h^{s}(x_0)} + \norm{b}_\infty \le L \} 
\end{equation*}
for $L > 0$. 

Several articles have discussed the richness of those classes of functions. 
In particular, Proposition 13 in~\cite{MaestraHoffmann2022} gives an example of Mckean-Vlasov homogeneous particle systems that fall into this class. 
With some slight modifications to its proof, we have the following result.
\begin{proposition}
	\label{pp:exist-holder-class}
	Let $\sigma = \sigma_0 I_{d \times d}$ for some $\sigma_0 > 0$. 
	Let $b(x,y) = V(x) + F(x-y)$ with $F, V \in C_c^1$ and  
	\begin{equation*}
		\norm{V}_{\h^s} + \norm{F}_{\h^{s'}} + \sup_{u \in I} \norm{\mu_{0,u}}_{\h^{s''}} < \infty \,,
	\end{equation*}
	for some $s, s' > 1$ with $s \notin \Z$, $s'' > 0$.
	Here, $\h^s$ denotes the global Hölder class (where we simply choose $U = \R^d$). 

	Suppose further that $\mu_{0,u}$ are probability measures with finite first moments and continuous density functions uniformly bounded over $u \in I$. 
	Then, for every $t_0 \in (0,T)$ and $x_0 \in \R^d$, we have $\mu_{t_0,u} \in \h^{s}(x_0)$ for all $u \in I$.
\end{proposition}

Now we present the minimax theorem, over the particle systems within those restricted smoothness classes. 

\begin{theorem}
	\label{t:minimax-mu}
	Let $L > 0$ and $s \in (0,1)$. 
	Assume one of the following holds: 
	\begin{enumerate}[\rm (a)]
		\item Hypothesis of Lemma~\ref{l:improve-est-mu}(1) and $s \ge \frac{1}{2}$,
		\item Hypothesis of Lemma~\ref{l:improve-est-mu}(2) with $p > 2$ and $s \in (0,\frac{1}{2})$ such that $p (2 - 4s) \le (p-2)d$. 
	\end{enumerate} 
	Then for every $t_0 \in (0, T)$, $u_0 \in (0, 1)$, $x_0 \in \R^d$, we have 
	\begin{equation}
		\label{e:minimax-mu-upper}
		\sup_{ (b,\sigma,G,\mu_0) \in \a^{s}_L (t_0,x_0) } \inf_{h_2, h_3 > 0} \E \abs{ \hat \mu^n_h (t_0,u_0,x_0) - \mu(t_0,u_0,x_0) }^2 \le C n^{-\frac{2s}{d+3s}} \,.
	\end{equation}
	Moreover, 
	\begin{equation}
		\label{e:minimax-mu-lower}
		\inf_{\hat \mu} \sup_{ (b,\sigma,G,\mu_0) \in \a^{s}_L (t_0,x_0) } \E \abs{ \hat \mu - \mu(t_0,u_0,x_0) }^2 \ge c n^{-\frac{2s}{d+3s}} \,,
	\end{equation}
	where the infimum is taken over all possible estimators of $\mu(t_0,u_0,x_0)$ constructed from $\mu^n_{t_0}$. 
	Both constants $C$ and $c$ depend only on the parameters $T, d, L$, the kernels $J,K$, the function $\rho_I$ given by Condition~\ref{cd:init-data}(3), and the values of $\mu$ in a small neighborhood of $t_0, u_0, x_0$.
\end{theorem}

The proofs are written in Section~\ref{s:minimax-mu-prf}.
Here we make several remarks on the above results. 
\begin{remark}
	Without the extra assumption on the $p$-Wasserstein continuity on initial data in Lemma~\ref{l:improve-est-mu}(2), 
	we would obtain a suboptimal upper bound when $s < \frac{1}{2}$, namely $n^{-\frac{1}{d+1+s}}$ (though we still attain the optimal bound when $s \ge \frac{1}{2}$). 
\end{remark}
\begin{remark}
	We mark that the asymptotic behavior is slower than that of~\cite{MaestraHoffmann2022}, namely $n^{-\frac{2s}{d+3s}}$ rather than $n^{-\frac{2s}{d+2s}}$. 
	This is due to the heterogeneity of our graphon particle system. 
	Both the index gap and the density gap between the target particle $u_0$ and the regularly spaced observations are $O(n^{-1})$, which actually reduces the approximation accuracy. 
	Yet it is possible to estimate the average density $\bar \mu = \int_I \mu_u du$ using exactly the same strategy as~\cite{MaestraHoffmann2022}, and that should probably exhibit the identical asymptotic behavior. 
\end{remark}
\begin{remark}
	Note that our algorithm to estimate $\mu$ is not adaptive to the observed data.
	Instead, users are free to set the parameters (e.g.\ bandwidths) according to their own accuracy demands, and the parameters are fixed from the start. 
	As long as the bandwidths are chosen appropriately, our estimator still achieves optimality. 
	It also improves computational efficiency compared to the data-driven Goldenshluger-Lepski algorithm applied in~\cite{MaestraHoffmann2022}, which selects the best bandwidths among the set of candidates by making $O(n)$-many comparisons. 
	Nevertheless, the adaptive estimator automatically fits the data with the best parameters and produces an error just a logarithmic factor higher than the lower bound.
	It is nonparametric and requires less knowledge about the initial state of the particle systems. 
\end{remark}

\subsection{Minimax estimation for drift}

In this section, we consider a slight generalization of the graphon mean-field system, where the drift coefficient is time-inhomogeneous. 
Namely, we extend~\eqref{e:graphon-SDE} to 
\begin{equation}
	\label{e:graphon-SDE-inhom}
	X_u(t) = X_u(0) + \int_0^t \int_I b(t,X_u(s),x) G(u,v) \mu_{s,v}(dy)dv ds + \int_0^t \sigma(X_u(s)) dB_u(s) \,,
\end{equation}
where $b \in C^1([0,T]; W^{2,\infty}(\R^d \times \R^d; \R^d) )$ satisfies that, $\forall t,t' \in [0,T]$, $ x,x',y,y' \in \R^d$,
\begin{equation}
	\label{e:b-lip-inhom}
	\abs{b(t,x,y) - b(t',x',y')} \le C(\abs{t-t'} + \abs{x-x'} + \abs{y-y'}) \,,
\end{equation}
for some constant $C > 0$ (compare with~\eqref{e:b-lip}). 
Then the drift term $\beta$ is given by 
\begin{equation*}
	\beta(t,u,x,\mu_t) = \int_I G(u,v) \int_{\R^d} b(t,x,y) \mu_{t,v}(dy) dv \,.
\end{equation*}
Note that the stability analysis in~\cite{BayraktarChakrabortyWu2023} and previous estimates (Lemmata~\ref{l:improve-est-mu} and~\ref{l:improve-est-beta}) still hold. 
We will conduct a minimax analysis on the drift term $\beta$ to show that the estimator $\hat \beta^n_{h,\kappa}$ is optimal in this general setting. 

We can extend the notion of Hölder continuity to the space of time-dependent functions. 
\begin{definition}
	Let $t_0 \in (0,T)$ and $x_0 \in \R^d$, and $s_1, s_3 > 0$. 
	We say that a function $f: (0, T) \times \R^d \to \R$ belongs to class $\h^{s_1,s_3}(t_0,x_0)$ if there exists an open neighborhood $U$ of $(t_0, x_0)$ in $(0,T) \times \R^d$ such that 
	\begin{equation*}
		\norm{f}_{\h^{s_1,s_3}(t_0,x_0)} \defeq \sup_{(t,x) \in U} \abs{f(t,x)} + C(f,U) < \infty \,,
	\end{equation*}
	where $C(f,U)$ is the minimum positive number $C$ such that 
	\begin{equation*}
		\abs{D^{\alpha_1, \alpha_3} f(t,x) - D^{\alpha_1, \alpha_3} f(s,y)} \le C (\abs{t-s}^{s_1 - \floor{s_1}} + \abs{x-y}^{s_3 - \floor{s_3}}) 
	\end{equation*}
	for all $(t,x), (s,y) \in U$ and all (multi-)indices $\alpha_1, \alpha_3$ with $\abs{\alpha_1} \le s_1$, $\abs{\alpha_3} \le s_3$. 

	We say that a function with values $\R^m$ $f = (f_1, \dots, f_m)$ belongs to class $\h^{s_1,s_3}(t_0,x_0)$ if each component $f_j \in \h^{s_1,s_3}(t_0,x_0)$. 
\end{definition}

Let $\check{\mathcal{P}}$ be the set of all parameters $(b,\sigma,G,\mu_0)$ satisfying the same conditions as $\mathcal{P}$ defined above Proposition~\ref{pp:exist-holder-class}, except that $b$ has the form $b(t,x,y)$ and satisfies~\eqref{e:b-lip-inhom}. 
For $s_1, s_3 > 0$, we define 
\begin{equation*}
	\check \a^{s_1,s_3}(t_0,u_0,x_0) = \{ (b,\sigma,G,\mu_0) \in \check{\mathcal{P}} \st \mu = S(b,\sigma,G,\mu_0), \mu_{u_0} \in \h^{s_1,s_3}(t_0,x_0) \} \,,
\end{equation*}
and set
\begin{equation*}
	\check \a^{s_1,s_3}(t_0,x_0) = \bigcap_{u \in I} \check \a^{s_1,s_3}(t_0,u,x_0) \,.
\end{equation*}
Moreover, we consider a restriction of this class 
\begin{equation*}
	\check \a^{s_1,s_3}_L(t_0,x_0) = \{ (b,\sigma,G,\mu_0) \in \check{\mathcal{P}} \st \norm{ S(b,\sigma,G,\mu_0) }_{\h^{s_1,s_3}(t_0,x_0)} + \norm{b}_{C^1_t W_{x,y}^{2,\infty}} \le L \} 
\end{equation*}
for $L > 0$. 

Then we present a minimax result on the estimation of the drift term $\beta$. 
The proof will be given in Section~\ref{s:minimax-beta-prf}.
\begin{theorem}
	\label{t:minimax-beta}
	Let $L > 0$, and $s_1, s_3 \in (0,1)$. 
	Define the effective smoothness $s_b$ by 
	\begin{equation*}
		\frac{1}{s_b} = \frac{1}{s_1} + 1 + \frac{1}{s_3} \,.
	\end{equation*}
	Assume the hypothesis of Lemma~\ref{l:improve-est-beta}, and $s_1, s_3$ satisfy the following conditions
	\begin{equation*}
		\frac{1}{s_1} - \frac{1}{s_3} + 2 \ge 0 \,, \qquad \frac{1}{s_1} - \frac{2}{s_3} + 4 \ge 0 \,.
	\end{equation*} 

	Then, for every $t_0 \in (0,T)$, $u_0 \in (0,1)$, $x_0 \in \R^d$, we have 
	\begin{equation}
		\label{e:minimax-beta-upper}
		\sup_{ (b,\sigma,G,\mu_0) \in \check \a^{s_1,s_3}_L (t_0,x_0) } \inf_{h_1,h_2,h_3,\kappa_2 > 0} \E \abs{ \hat \beta^n_{h,\kappa} (t_0,u_0,x_0) - \beta(t_0,u_0,x_0) }^2 \le C n^{-\frac{2s_b}{2s_b+1}} \,.
	\end{equation}
	Moreover
	\begin{equation}
		\label{e:minimax-beta-lower}
		\inf_{\hat \beta} \sup_{ (b,\sigma,G,\mu_0) \in \check \a^{s_1,s_3}_L (t_0,x_0) } \E \abs{ \hat \beta - \beta(t_0,u_0,x_0) }^2 \ge c n^{-\frac{2s_b}{2s_b+1}} \,,
	\end{equation}
	where the infimum is taken over all possible estimators $\hat \beta$ built on empirical data $(\mu^n_t)_{t \in [0,T]}$. 
	Both constants $C$ and $c$ depend only on the parameters $T, d, L, \kappa_2$, the kernels $H,J,K$, the function $\rho_I$ given by Condition~\ref{cd:init-data}(3), and the values of $\mu$ in a small neighborhood of $t_0, u_0, x_0$. 
\end{theorem}

\begin{remark}
	It is also possible to obtain a similar argument in the case $s_3 < \frac{1}{2}$, subject to some additional assumptions on the $\was_p$-continuity of the initial data, as in Theorem~\ref{t:minimax-mu}. 
	We leave this to the readers. 
\end{remark}

\begin{remark}
	As mentioned above, the stability of graphon particle systems given in~\cite{BayraktarChakrabortyWu2023} and the consistency of the estimations in the Lemmata~\ref{l:improve-est-mu} and~\ref{l:improve-est-beta} remain true in the above generalization with the time-inhomogeneous drift coefficient $b$. 
	In this case, the time dependence of $\beta = \beta(t,u,x,\mu_t)$ is not only dependent on the mean field but also on an independent variable $t$. 
	This gives more freedom in constructing examples and thus opens up a cleaner path towards the theoretical lower bound. 
	It is for that reason that we include the time-inhomogeneous drifts for the minimax analysis. 
\end{remark}

\subsubsection*{Optimality of graphon estimator $\hat G$}
Recall from~\eqref{e:estGdef} and the proof of Theorem~\ref{t:main} that the error of our estimator $\hat G$ depends on the $L^2$ variation of estimators $\hat \mu$ and $\hat \beta$. 
There are several convergent quantities that require some stronger conditions to make them explicitly quantitative, but we are not diving into the details in this work. 
There is also a particular item, $\theta_1(\tilde r, \kappa_1)$ in the proof of Theorem~\ref{t:main}, whose convergence is due to Assumption~\ref{as:FLmu}, that is not explicitly quantifiable without further assumptions.  
This keeps us away from a minimax analysis of the error $\E \abs{\hat G - G}^2$ to study its optimality. 
However, given the pointwise optimality of $\hat \mu$ and $\hat \beta$, our estimators should show relatively good performances.

\section{Proofs for Section~\ref{s:estimates}}\label{s:estimates-proof}

\subsection{Proof of Lemma~\ref{l:JK-est-mu} and Corollary~\ref{c:err-bound-mu-overall}}

\begin{proof}[Proof of Lemma~\ref{l:JK-est-mu}]
	Fix $t_0, u_0, x_0$.
	Recall that 
	\begin{equation*}
		\hat \mu^n_h (t_0, u_0, x_0) = \JK_h \ast \mu^n_{t_0} = \frac{1}{n} \sum_{i=1}^n J_{h_2} (u_0 - \frac{i}{n}) K_{h_3} (x_0 - X^n_i(t_0)) \,.
	\end{equation*}
	We do the following inequality via a telescoping sum,
	\begin{align*}
		& \E \abs{\hat \mu^n_h (t_0, u_0, x_0) - \mu(t_0, u_0, x_0)}^2 \le \\
		& \qquad 4 \E \abs{ \frac{1}{n} \sum_{i=1}^n J_{h_2}(u_0 - \frac{i}{n}) \paren[\big]{ K_{h_3}(x_0 - X^n_i(t_0)) - K_{h_3}(x_0 - X_{\frac{i}{n}}(t_0)) } }^2 \\
		& \quad + 4 \E \abs{ \frac{1}{n} \sum_{i=1}^n J_{h_2}(u_0 - \frac{i}{n}) \paren[\big]{ K_{h_3}(x_0 - X_{\frac{i}{n}}(t_0)) - \E K_{h_3}(x_0 - X_{\frac{i}{n}}(t_0))  } }^2 \\
		& \quad + 4 \E \abs{ \int_I J_{h_2} (u_0 - \frac{\ceil{nu}}{n}) \E (K_{h_3}(x_0 - X_{\frac{i}{n}}(t_0))) - J_{h_2}(u_0 - u) \E(K_{h_3}(x_0 - X_{u}(t_0))) du }^2 \\
		& \quad + 4 \E \abs{ \JK_h \ast \mu (t_0, u_0, x_0) - \mu (t_0, u_0, x_0) }^2 \\
		& \quad =: 4 (M_1 + M_2 + M_3 + M_4) \,.
	\end{align*}
	We only need to provide the appropriate bounds for $M_1$, $M_2$, and $M_3$. 

	\textit{Step 1.} We bound $M_1$ using the convergence of the finite-population system to the graphon mean-field system.
	\begin{align*}
		M_1 & \le \frac{1}{n} \sum_{i=1}^n \abs{J_{h_2}(u_0 - \frac{i}{n})}^2 \E \abs{ K_{h_3}(x_0 - X^n_i(t_0)) - K_{h_3}(x_0 - X_{\frac{i}{n}}(t_0)) }^2 \\
		& \le \frac{1}{n} \sum_{i=1}^n J_{h_2}(u_0 - \frac{i}{n})^2 \norm{\grad K_{h_3}}_\infty^2 \E \abs{ X^n_i(t_0) - X_{\frac{i}{n}}(t_0) }^2 \,.
	\end{align*}
	Then applying Lemma~\ref{l:conv-of-system} gives the bound 
	\begin{equation*}
		M_1 \lesssim n^{-2} h_3^{-2-2d} \norm{\grad K}_\infty^2 \sum_{i=1}^n J_{h_2}(u_0 - \frac{i}{n})^2 \,.
	\end{equation*}

	\textit{Step 2.} We bound $M_2$ following the idea of~\cite{MaestraHoffmann2022}.

	For $i = 1, \dots, n$, let 
	\begin{equation*}
		Z_i = J_{h_2}(u_0-\frac{i}{n}) \paren[\big]{ K_{h_3}(x_0 - X_{\frac{i}{n}}(t_0)) - \E (K_{h_3}(x_0 - X_{\frac{i}{n}}(t_0))) } \,.
	\end{equation*}
	The second part simplifies to 
	\begin{equation*}
		M_2 = \E \abs{\frac{1}{n} \sum_{i=1}^n Z_i}^2 = \int_0^\infty \P \left( \abs{ \frac{1}{n} \sum_{i=1}^n Z_i } > \sqrt{z} \right) dz \,.
	\end{equation*}

	From~\eqref{e:graphon-SDE} we know $X_u$ are autonomous. 
	And due to the independence of the Brownian motions $\{B_u: u \in I\}$, $Z_1, \dots, Z_n$ are all independent. 
	We have $\E Z_i = 0$ and $\abs{Z_i} \le \norm{\JK_h}_\infty < \infty$ for every $i = 1, \dots, n$.
	Then Bernstein's inequality reads
	\begin{equation*}
		\P \left( \abs{ \frac{1}{n} \sum_{i=1}^n Z_i } > \sqrt{z} \right) \le 2 \exp \left( - \frac{\frac{1}{2}n^2 z}{\sum_{i=1}^n \E Z_i^2 + \frac{1}{3}n \sqrt{z} \norm{\JK_h}_\infty} \right) \,.
	\end{equation*}
	Now we apply the inequality (48) in~\cite{MaestraHoffmann2022} to see that 
	\begin{equation}
		\label{e:from-48}
		\int_0^\infty \P \left( \abs{ \frac{1}{n} \sum_{i=1}^n Z_i } > \sqrt{z} \right) dz \lesssim \max \left\{ 2 n^{-2}\sum_{i=1}^n \E Z_i^2, \frac{4}{9} n^{-2} \norm{\JK_h}_\infty^2 \right\} \,.
	\end{equation}

	Observe that 
	\begin{equation*}
		\E Z_i^2 \le 4 J_{h_2}(u_0-\frac{i}{n})^2 \E (K_{h_3}(x_0 - X_{\frac{i}{n}}(t_0))^2) = 4 J_{h_2}(u_0-\frac{i}{n})^2 \norm{K_{h_3}(x_0-\cdot)}_{L^2(\mu_{t_0,\frac{i}{n}})}^2 \,.
	\end{equation*}
	Plugging into~\eqref{e:from-48} gives the bound 
	\begin{equation*}
		M_2 \lesssim n^{-2} \sum_{i=1}^{n} J_{h_2} (u_0 - \frac{i}{n})^2 \norm{K_{h_3} (x_0 - \cdot)}^2_{L^2(\mu_{t_0, \frac{i}{n}})} + n^{-2} h_2^{-2} h_3^{-2d} \norm{J}_\infty^2 \norm{K}_\infty^2  \,.
	\end{equation*}

	\textit{Step 3.}
	Notice that $J_{h_2}$ is supported on $\overline{B(0,h_2)}$. 
	We bound $M_3$ using only Minkowski's inequality and the inequality mean-value theorem. 
	\begin{align*}
		M_3 & \le 4h_2 \int_I \abs{ J_{h_2}(u_0-\frac{\ceil{nu}}{n}) - J_{h_2}(u_0-u) }^2  (\E(K_{h_3}(x_0 - X_{\frac{\ceil{nu}}{n}}(t_0))))^2 \\
		& \qquad + J_{h_2}(u_0-u)^2 \paren[\big]{ \E \abs{ K_{h_3}(x_0 - X_{\frac{\ceil{nu}}{n}}(t_0)) - K_{h_3}(x_0 - X_{u}(t_0)) }^2 }^2 du \\
		& \le 4h_2 \int_{2 (u_0 + \supp(J_{h_2}))} \norm{\grad J_{h_2}}_\infty^2 \abs{ \frac{\ceil{nu}}{n} - u }^2 \norm{K_{h_3}(x_0 - \cdot)}_{L^2(\mu_{t_0,\frac{\ceil{nu}}{n}})}^2 \\
		& \qquad + J_{h_2}(u_0-u)^2 \norm{\grad k_{h_3}}_\infty^2 \E \abs{ X_{\frac{\ceil{nu}}{n}}(t_0) - X_u(t_0) }^2 du \\
		& \lesssim n^{-3} h_2^{-2} \norm{\grad J}_\infty^2 \sum_{i=1}^n \norm{K_{h_3}(x_0 - \cdot)}_{L^2(\mu_{t_0,\frac{i}{n}})}^2 + n^{-2} h_3^{-2-2d} \norm{J}_2^2 \norm{\grad K}_\infty^2 \,,
	\end{align*}
	where the last step uses Theorem 2.1 of~\cite{BayraktarChakrabortyWu2023}.
	That finishes the proof. 
\end{proof}

We will consistently use the equality that follows from Fubini-Tonelli theorem,
\begin{equation}
	\label{e:from-Fubini}
	\int_{\R^d} \norm{K_{h_3}(x - \cdot)}_{L^2(\mu_{t,u})}^2 dx = h_3^{-d} \norm{K}_2^2 \,,
\end{equation}
where the latter $L^2$-norm is taken with respect to the Lebesgue measure on $\R^d$. 

\begin{proof}[Proof of Corollary~\ref{c:err-bound-mu-overall}]
	Recall that $\hat \mu^n_{h,r} = \hat \mu^n_h \one{\{\abs{x}\le r\}}$. 
	We break the integral into two parts 
	\begin{align*}
		&\int_{\tau_1}^{\tau_2} \int_I \int_{\R^d} \E \abs{\hat \mu^n_{h,r}(t, u, x) - \mu(t, u, x)}^2 dx du dt = \\
		& \qquad \int_{\tau_1}^{\tau_2} \int_I \int_{\{\abs{x} \le r\}} \E \abs{\hat \mu^n_{h,r}(t, u, x) - \mu(t, u, x)}^2 dx du dt + \int_{\tau_1}^{\tau_2} \int_I \int_{\{\abs{x} > r\}} \mu(t, u, x)^2 dx du dt \,.
	\end{align*}

	The second part tends to 0 as $r \to \infty$ due to Proposition~\ref{pp:mu-L2-bound} and dominated convergence. 
	We denote the convergence rate by $\theta_{2,\mu}(r)$. 

	For the first part, we rearrange and integrate the terms given in Lemma~\ref{l:JK-est-mu}.
	Recall that 
	\begin{align}
		\nonumber
		& \E \abs{\hat \mu^n_h (t, u, x) - \mu(t, u, x)}^2 \lesssim \\
		\nonumber
		& \qquad n^{-2} h_2^{-2} h_3^{-2d} + n^{-2}  h_3^{-2-2d} \\
		\label{e:cmu-l2}
		& \quad + n^{-2} h_3^{-2-2d} \norm{\grad K}_\infty^2 \sum_{i=1}^n J_{h_2}(u - \frac{i}{n})^2 \\
		\label{e:cmu-l3}
		& \quad + n^{-2} \sum_{i=1}^{n} J_{h_2} (u - \frac{i}{n})^2 \norm{K_{h_3} (x - \cdot)}^2_{L^2(\mu_{t, \frac{i}{n}})} \\
		\label{e:cmu-l4}
		& \quad + n^{-3} h_2^{-2} \norm{\grad J}_\infty^2 \sum_{i=1}^{n} \norm{K_{h_3}(x - \cdot)}_{L^2(\mu_{t_, \frac{i}{n}})}^2  \\
		\nonumber
		& \quad + \abs{\JK_h \ast \mu_{t} (u, x) - \mu(t, u, x)}^2 \,.
	\end{align}
	The first line of bounds are independent of $t, u, x$, so that integrating them gives 
	\begin{equation*}
		T r^d (n^{-2} h_2^{-2} h_3^{-2d} + n^{-2} h_3^{-2-2d}) \,.
	\end{equation*}

	The middle three lines are computed as follows.
	For line~\eqref{e:cmu-l2} we have
	\begin{align*}
		& \int_I \int_{\{\abs{x} \le r\}} n^{-2} h_3^{-2-2d} \norm{\grad K}_\infty^2 \sum_{i=1}^n J_{h_2}(u - \frac{i}{n})^2 dx du \\
		& \qquad \le r^d n^{-2} h_3^{-2-2d} \norm{\grad K}_\infty^2 \sum_{i=1}^n \int_\R J_{h_2}(u - \frac{i}{n})^2 \\
		& \qquad = r^d n^{-1} h_2^{-1} h_3^{-2-2d} \norm{J}_2^2 \,.
	\end{align*}
	With the same idea and using~\eqref{e:from-Fubini}, line~\eqref{e:cmu-l3} gives 
	\begin{align*}
		& \int_I \int_{\R^d} n^{-2} \sum_{i=1}^{n} J_{h_2} (u - \frac{i}{n})^2 \norm{K_{h_3} (x - \cdot)}^2_{L^2(\mu_{t, \frac{i}{n}})} dx du \\
		& \qquad \le n^{-1} h_2^{-1} h_3^{-d} \norm{J}_2^2 \norm{K}_2^2 \,.
	\end{align*}
	Analogously, line~\eqref{e:cmu-l4} gives
	\begin{align*}
		& \int_I \int_{\R^d} n^{-3} h_2^{-2} \norm{\grad J}_\infty^2 \sum_{i=1}^{n} \norm{K_{h_3}(x - \cdot)}_{L^2(\mu_{t_, \frac{i}{n}})}^2 dx du \\
		& \qquad = n^{-2} h_2^{-2} h_3^{-d} \norm{\grad J}_\infty^2 \norm{K}_2^2 \,.
	\end{align*}

	In addition, the final line expands as 
	\begin{align*}
		& \int_I \int_{\R^d} \abs{\JK_h \ast \mu_{t} (u, x) - \mu(t, u, x)}^2 du dx \\
		& \le \int_{\R \times \R^d} \abs{ \int_{\R \times U} \JK_h (u',x') (\mu(t,u-u',x-x') - \mu(t,u,x)) dx' du' }^2 dx du \\
		& \le \left( \int_{\R \times \R^d} \JK_h(u',x') \left( \int_{\R \times \R^d} (\mu(t,u-u',x-x') - \mu(t,u,x))^2 dxdu \right)^{1/2} dx' du' \right)^{2} \\
		& \le \sup_{\abs{u'} \le h_2, \abs{x'} \le h_3} \norm{\mu(t,\cdot-u',\cdot-x') - \mu(t,\cdot,\cdot)}_{L^2(\R \times \R^d)}^2 \,.
	\end{align*}
	Note that $\norm{\mu_{t,u}}_2$ is uniformly bounded for $t \in [\tau_1, \tau_2]$ and $u \in I$ as a consequence of Proposition~\ref{pp:mu-L2-bound} (see also Corollary 8.2.2 of~\cite{BogachevKrylovRocknerShaposhnikov2015} for details on local upper bounds of particle density). 
	As translations converge in $L^2$, with dominated convergence we have 
	\begin{equation}
		\label{e:JK-conv-L2}
		\lim_{h_2, h_3 \to 0} \int_{\tau_1}^{\tau_2} \int_I \int_{\R^d} \abs{\JK_h \ast \mu_{t} (u, x) - \mu(t, u, x)}^2 du dx dt = 0 \,.
	\end{equation}
	We denote the convergence rate by $\theta_{3,\mu}(h)$. 

	In summary, the $L^2$-error of $\hat \mu^n_{h,r}$ is given by
	\begin{align*}
		& \int_{\tau_1}^{\tau_2} \int_I \int_{\R^d} \E \abs{\hat \mu^n_{h,r}(t, u, x) - \mu(t, u, x)}^2 dx du dt \lesssim_{T,b,J,K} \theta_{2,\mu}(r) +  \theta_{3,\mu}(h) + \\
		& \qquad r^d (n^{-2} h_3^{-2-2d} + n^{-2} h_2^{-2} h_3^{-2d}) + n^{-1} h_2^{-1}  h_3^{-d} + n^{-2} h_2^{-2} h_3^{-d} \,.
	\end{align*}
\end{proof}

\subsection{Proof of Lemma~\ref{l:HJK-est-pi} and Corollary~\ref{c:err-bound-beta-overall}}

Recall the dynamics of $X_u$ defined in~\eqref{e:graphon-SDE}
For simplicity, we let 
\begin{equation*}
	Y_u(t) \defeq \beta(t,u,X_u(t)) = \int_I \int_{\R^d} b(X_u(t),x) G(u,v) \mu_{t,v}(dy) dv 
\end{equation*}
for every $u \in I$ and $t \in [0, T]$. 
Similarly, for the finite-population system, we let 
\begin{equation*}
	Y^n_i(t) \defeq \frac{1}{n} \sum_{j=1}^n b(X^n_i(t), X^n_j(t)) g^n_{ij}
\end{equation*}
for $i = 1, \dots, n$, and $t \in [0, T]$. 
Observe that $\abs{Y^n_i(t)}, \abs{Y_u(t)} \le \norm{b}_\infty$. 
We have the following consistency result. 
\begin{lemma}
	\label{l:approx-Yu}
	We assume the same hypothesis as in Lemma~\ref{l:conv-of-system}.
	In the $n$-particle system, we have the following
	\begin{equation}
		\max_{1 \le i \le n} \E \abs{Y^n_i(t) - Y_{\frac{i}{n}}(t)}^2 \lesssim \max_{1 \le i \le n} \E \abs{X^n_i(t) - X_{\frac{i}{n}}(t)}^2 + \frac{1}{n} \,,
	\end{equation}
	for every $t \in [0, T]$. 
	As a consequence,
	\begin{equation*}
		\max_{i=1,\dots,n} \int_0^T \psi(t) \E \abs{Y_{\frac{i}{n}}(t) - Y^n_{i}(t)}^2 \lesssim_\psi \frac{1}{n} \,.
	\end{equation*}
	for any bounded continuous function $\psi$. 
\end{lemma}

Although we defer the proof to Appendix~\ref{s:pf-technical}, we are now ready to prove our estimates of $\pi$.

\begin{proof}[Proof of Lemma~\ref{l:HJK-est-pi}]
	Fix $t_0, u_0, x_0$. 
	Recall that 
	\begin{align*}
		\hat \pi^n_h (t_0, u_0, x_0) & = \int_0^T \frac{1}{n} \sum_{i=1}^n \HJK_h(t_0-t, u_0-\frac{i}{n}, x_0-X^n_i(t)) d X^n_i(t) \\
		& = \int_0^T \frac{1}{n} \sum_{i=1}^n \HJK_h(t_0-t, u_0-\frac{i}{n}, x_0-X^n_i(t)) Y^n_i(t) dt \\
		& \qquad + \int_0^T \frac{1}{n} \sum_{i=1}^n \HJK_h(t_0-t, u_0-\frac{i}{n}, x_0-X^n_i(t)) \sigma(X^n_i(t)) dB_{\frac{i}{n}}(t) \,.
	\end{align*}
	From this we may then write  
	\begin{align*}
		& \E \abs{\pi^n_h(t_0,u_0,x_0) - \pi(t_0,u_0,x_0)}^2 \le \\
		& \qquad 5 \E \left\vert \int_0^T H_{h_1}(t_0-t) \frac{1}{n} \sum_{i=1}^n J_{h_2}(u_0-\frac{i}{n}) \right.  \\
		& \qquad \qquad \qquad  \left. \paren[\big]{ K_{h_3}(x_0 - X^n_i(t)) Y^n_i(t) - K_{h_3}(x_0 - X_{\frac{i}{n}}(t)) Y_{\frac{i}{n}}(t) } dt \right\vert^2 \\
		& \quad + 5 \E \left\vert \int_0^T H_{h_1}(t_0-t) \frac{1}{n} \sum_{i=1}^n J_{h_2}(u_0-\frac{i}{n}) \right. \\ 
		& \qquad \qquad \qquad  \left. \paren[\big]{ K_{h_3}(x_0 - X_{\frac{i}{n}}(t)) Y_{\frac{i}{n}}(t) - \E (K_{h_3}(x_0 - X_{\frac{i}{n}}(t)) Y_{\frac{i}{n}}(t)) } dt \right\vert^2 \\
		& \quad + 5 \E \left\vert \int_0^T H_{h_1}(t_0-t) \int_I \left( J_{h_2}(u_0-\frac{\ceil{nu}}{n}) \E (K_{h_3}(x_0 - X_{\frac{\ceil{nu}}{n}}(t)) Y_{\frac{\ceil{nu}}{n}}(t)) \right. \right. \\
		& \qquad \qquad \qquad  \left. \left. - J_{h_2}(u_0-u) \E (K_{h_3}(x_0-X_u(t)) Y_u(t)) \right) dudt \right\vert^2 \\
		& \quad + 5 \E \abs{ \int_0^T \int_I \int_{\R^d} \HJK_h(t_0-t,u_0-u,x_0-x) \pi(t,u,x) dxdudt - \pi(t_0,u_0,x_0) }^2 \\
		& \quad + 5 \E \abs{ \int_0^T \frac{1}{n} \sum_{i=1}^n H_{h_1}(t_0-t) J_{h_2}(u_0-\frac{i}{n}) K_{h_3}(x_0-X^n_i(t)) \sigma(X^n_i(t)) dB_{\frac{i}{n}}(t) }^2 \\
		& \quad =: 5 (P_1 + P_2 + P_3 + P_4 + P_5) \,.
	\end{align*}
	We do not need to do anything for $P_4$ for now, and $P_5$ will be proved using standard techniques in stochastic analysis at the end. 
	The proof of the rest are bounded in analogously as $M_1$ through $M_3$ in the proof of Lemma~\ref{l:JK-est-mu}.

	\textit{Step 1.}
	Observe that $P_1$ is upper bounded by
	\begin{align*}
		& T \int_0^T H_{h_1}^2(t_0-t) \frac{1}{n} \sum_{i=1}^n J_{h_2}(u_0-\frac{i}{n})^2 \\
		& \qquad \qquad \E \abs{ K_{h_3}(x_0-X^n_i(t)) Y^n_i(t) - K_{h_3}(x_0-X_{\frac{i}{n}}(t)) Y_{\frac{i}{n}}(t) }^2 dt \,.
	\end{align*}
	For each $t$ and $i$, we have 
	\begin{align*}
		& \abs{ K_{h_3}(x_0-X^n_i(t)) Y^n_i(t) - K_{h_3}(x_0-X_{\frac{i}{n}}(t)) Y_{\frac{i}{n}}(t) } \\
		& \le \abs{K_{h_3}(x_0-X^n_i(t)) - K_{h_3}(x_0-X_{\frac{i}{n}}(t)) } \abs{Y^n_i(t)} + K_{h_3}(x_0-X_{\frac{i}{n}}(t)) \abs{Y^n_i(t) - Y_{\frac{i}{n}}(t)} \\
		& \le \norm{\grad K_{h_3}}_\infty \abs{X^n_i(t) - X_{\frac{i}{n}}(t)} \norm{b}_\infty + \norm{K_{h_3}}_\infty \abs{Y^n_i(t) - Y_{\frac{i}{n}}(t)} \,,
	\end{align*}
	so 
	\begin{align*}
		& \E \abs{ K_{h_3}(x_0-X^n_i(t)) Y^n_i(t) - K_{h_3}(x_0-X_{\frac{i}{n}}(t)) Y_{\frac{i}{n}}(t) }^2 \\
		& \le 2 \norm{\grad K_{h_3}}_\infty^2 \norm{b}_\infty^2 \E \abs{X^n_i(t) - X_{\frac{i}{n}}(t)}^2 + 2 \norm{K_{h_3}}_\infty^2 \E \abs{Y^n_i(t) - Y_{\frac{i}{n}}(t)}^2 \,.
	\end{align*}
	Then, with Lemmata~\ref{l:conv-of-system} and~\ref{l:approx-Yu}, we get
	\begin{equation*}
		P_1 \lesssim_{b,H,J,K,T} n^{-1} h_1^{-1} h_2^{-2} h_3^{-2-2d} + T n^{-2} h_3^{-2d} \sum_{i=1}^n J_{h_2}(u_0-u)^2 \,.
	\end{equation*}

	\textit{Step 2.}
	To bound $P_2$, we apply Berstein's inequality. 
	Let 
	\begin{equation*}
		Z_i(t) = J_{h_2}(u_0 - \frac{i}{n}) \paren[\big]{ K_{h_3}(x_0 - X_{\frac{i}{n}}(t)) Y_{\frac{i}{n}}(t) - \E(K_{h_3}(x_0 - X_{\frac{i}{n}}(t)) Y_{\frac{i}{n}}(t))  } \,.
	\end{equation*}
	Then  
	\begin{align*}
		P_2 & \le T\int_0^T H_{h_1}(t_0-t)^2 \E \abs{\frac{1}{n} \sum_{i=1}^n Z_i(t)}^2 dt \\
		& = T \int_0^T H_{h_1}(t_0-t)^2 \int_0^\infty \P \left( \abs{ \sum_{i=1}^n Z_i(t) } > n \sqrt{z} \right) dz dt \,.
	\end{align*}

	Observe that $\E Z_i(t) = 0$, and $\abs{Z_i(t)} \le 2 \norm{\JK_h}_\infty \norm{b}_\infty$ for every $i = 1, \dots, n$, for every $t$. 
	Also, every $Z_i(t)$ is a function of $X_{\frac{i}{n}}(t)$, which makes them independent of each other among $i = 1, \dots, n$. 
	So we may apply Bernstein's inequality and inequality (48) in~\cite{MaestraHoffmann2022} to obtain 
	\begin{align*}
		\P \left( \abs{ \sum_{i=1}^n Z_i(t) } > n \sqrt{z} \right) & \le 2 \exp \left( - \frac{\frac{1}{2}n^2 z}{\sum_{i=1}^n \E \abs{Z_i(t)}^2 + \frac{n \sqrt{z}}{3} 2 \norm{\JK_h}_\infty \norm{b}_\infty} \right) \\
		& \lesssim \max \left\{ 2 n^{-2} \sum_{i=1}^n \E \abs{Z_i(t)}^2 , \frac{16}{9} n^{-2} \norm{\JK_h}_\infty^2 \norm{b}_\infty^2 \right\} \,.
	\end{align*}
	Further notice that 
	\begin{equation*}
		\E \abs{Z_i(t)}^2 \le \norm{b}_\infty^2 J_{h_2}(u_0-\frac{i}{n})^2 \norm{K_{h_3}(x_0-\cdot)}_{L^2(\mu_{t,\frac{i}{n}})}^2 \,.
	\end{equation*}
	Thus 
	\begin{align*}
		P_2 & \lesssim T n^{-2} \norm{b}_\infty^2 \int_0^T H_{h_1}(t_0-t)^2 \sum_{i=1}^n J_{h_2}(u_0-\frac{i}{n})^2 \norm{K_{h_3}(x_0-\cdot)}_{L^2(\mu_{t,\frac{i}{n}})}^2 \\
		& \qquad + T n^{-2} h_1^{-1} h_2^{-2} h_3^{-2d} \norm{b}_\infty^2 \norm{H}_2^2 \norm{J}_\infty^2 \norm{K}_\infty^2 \,.
	\end{align*}

	\textit{Step 3.}
	The idea for bounding $P_3$ is analogous to that of $M_3$ in the proof of Lemma~\ref{l:JK-est-mu}, which uses the stability of the graphon mean-field system.
	Observe that 
	\begin{equation*}
		P_3 \le T \int_0^T H_{h_1}(t_0-t)^2 4 h_2 \int_I \abs{ \E P_3(t,u) }^2 du dt \,,
	\end{equation*}
	where 
	\begin{equation*}
		P_3(t,u) = J_{h_2}(u_0 - \frac{\ceil{nu}}{n})  K_{h_3}(x_0-X_{\frac{\ceil{nu}}{n}}(t)) Y_{\frac{\ceil{nu}}{n}}(t)  - J_{h_2}(u_0-u) K_{h_3}(x_0-X_u(t)) Y_u(t) \,.
	\end{equation*}
	Note that 
	\begin{align*}
		\abs{P_3(t,u)}^2 & \le 2 \abs{ J_{h_2}(u_0 - \frac{\ceil{nu}}{n}) - J_{h_2}(u_0-u) }^2 K_{h_3}(x_0 - X_{\frac{\ceil{nu}}{n}}(t))^2 \abs{Y_{\frac{\ceil{nu}}{n}}(t)}^2 \\
		& \qquad + 2 J_{h_2}(u_0-u)^2 \abs{ K_{h_3}(x_0-X_{\frac{\ceil{nu}}{n}}(t)) Y_{\frac{\ceil{nu}}{n}}(t) - K_{h_3}(x_0-X_u(t)) Y_u(t) }^2 \\
		& \le 2 \norm{\grad J_{h_2}}_\infty^2 \abs{ \frac{\ceil{nu}}{n} - u }^2 K_{h_3}(x_0 - X_{\frac{\ceil{nu}}{n}}(t))^2 \norm{b}_\infty^2 \\
		& \qquad + 4 J_{h_2}(u_0-u)^2 K_{h_3}(x_0 - X_{\frac{\ceil{nu}}{n}}(t))^2 \abs{ Y_{\frac{\ceil{nu}}{n}}(t) - Y_u(t) }^2 \\
		& \qquad + 4 J_{h_2}(u_0-u)^2 \abs{ K_{h_3}(x_0 - X_{\frac{\ceil{nu}}{n}}(t)) - K_{h_3}(x_0 - X_{u}(t)) }^2 \abs{Y_u(t)}^2 \\
		& \le 2 n^{-2} h_2^{-2} \norm{\grad J}_\infty^2 \norm{b}_\infty^2 K_{h_3}(x_0 - X_{\frac{\ceil{nu}}{n}}(t))^2 \\
		& \qquad + 4 J_{h_2}(u_0-u)^2 h_3^{-2d} \norm{K}_\infty^2 \abs{ Y_{\frac{\ceil{nu}}{n}}(t) - Y_u(t) }^2 \\
		& \qquad + 4 J_{h_2}(u_0-u)^2 h_3^{-2-2d} \norm{\grad K}_\infty^2 \abs{ X_{\frac{\ceil{nu}}{n}}(t) - X_u(t) }^2 \norm{b}_\infty^2 \,.
	\end{align*}
	Then
	\begin{align*}
		\E \abs{P_3(t,u)}^2 & \le 2 n^{-2} h_2^{-4} \norm{\grad J}_\infty^2 \norm{b}_\infty^2 \norm{K_{h_3}(x_0-\cdot)}_{L^2(\mu_{t,\frac{\ceil{nu}}{n}})}^2 \\
		& \quad + 4 n^{-2} h_3^{-2d} \norm{K}_\infty^2 J_{h_2}(u_0-u)^2 + 4 n^{-2} h_3^{-2-2d} \norm{\grad K}_\infty^2 \norm{b}_\infty^2 J_{h_2}(u_0-u)^2 \,.
	\end{align*}
	Integrating those produces
	\begin{align*}
		P_3 & \lesssim T n^{-2} h_2^{-2} \norm{\grad J}_\infty^2 \norm{b}_\infty^2 \int_0^T H_{h_1}(t_0-t)^2 \int_I \norm{K_{h_3}(x_0-\cdot)}_{L^2(\mu_{t,\frac{\ceil{nu}}{n}})}^2 du dt \\
		& \qquad + T n^{-2} h_1^{-1} h_3^{-2-2d} \norm{H}_2^2 \norm{J}_2^2 \norm{\grad K}_\infty^2 \norm{b}_\infty^2  \\
		& \qquad + T n^{-2} h_1^{-1} h_3^{-2d} \norm{H}_2^2 \norm{J}_2^2 \norm{K}_\infty^2 \,.
	\end{align*}

	\textit{Step 4.}
	For $P_5$, notice that $\{B_{\frac{i}{n}} \st i = 1, \dots, n\}$ are distinct independent Brownian motions. 
	Then we apply It\^o's isometry to see that 
	\begin{align*}
		P_5 & = \frac{1}{n^2} \E \abs{ \int_0^T \sum_{i=1}^n \HJK_h (t_0-t, u_0-\frac{i}{n}, x_0-X^n_i(t)) \sigma(X^n_i(t)) dB_{\frac{i}{n}}(t) }^2 \\
		& = \frac{d}{n^2} \E \left( \int_0^T \sum_{i=1}^n \HJK_h(t_0-t, u_0-\frac{i}{n}, x_0-X^n_i(t))^2 \tr(\sigma \sigma^T) (X^n_i(t)) dt \right) \\
		& \le \frac{\sigma_+^2 d^2}{n^2} \sum_{i=1}^n \int_0^T \E  \HJK_h(t_0-t, u_0-\frac{i}{n}, x_0-X^n_i(t))^2 \\
		& \lesssim T d^2 \sigma_+^2 n^{-1} h_1^{-2} h_2^{-2} h_3^{-2d} \,.
	\end{align*}
	Adding all the above bounds finishes the proof. 
\end{proof}

\begin{proof}[Proof of Corollary~\ref{c:err-bound-beta-overall}]
	Recall that $\hat \beta^n_{h,\kappa,r} = \hat \beta^n_{h,\kappa} \one{\{\abs{x} \le r\}}$. 
	We break the integral into two parts
	\begin{align}
		\nonumber
		& \int_{\tau_1}^{\tau_2} \int_I \int_{\R^d} \E \abs{ \hat \beta^n_{h,\kappa,r}(t,u,x) - \beta(t,u,x) }^2 \\
		\label{e:telescope-beta-L2}
		& \quad = \int_{\tau_1}^{\tau_2} \int_I \int_{\{\abs{x} \le r\}} \E \abs{ \hat \beta^n_{h,\kappa}(t,u,x) - \beta(t,u,x) }^2 + \int_{\tau_1}^{\tau_2} \int_I \int_{\{\abs{x}>r\}} \abs{\beta(t,u,x) }^2 dx du dt \,.
	\end{align}

	\textit{Step 1.}
	The convergence of the second part is due to the $L^2$-integrability of $\beta$. 
	More precisely, recall that 
	\begin{equation*}
		\beta(t,u,x) = \int_I \int_{\R^d} b(x,y) G(u,v) \mu_{t,v}(dy) dv \,,
	\end{equation*}
	where $b(x,y) = F(x-y) + V(x)$ with $F, V \in L^1 \cap L^2 \cap L^\infty$. 
	Then 
	\begin{equation*}
		\abs{\beta(t,u,x)}^2 \le 2 \abs{V(x)}^2 + 2 \int_{\R^d} \abs{F(x-y)}^2 \mu_{t,v}(dy) \,,
	\end{equation*}
	so that 
	\begin{equation*}
		\int_{\tau_1}^{\tau_2} \int_I \int_{\R^d} \abs{\beta(t,u,x)}^2 \le 2T (\norm{V}_2^2 + \norm{F}_2^2) < \infty\,.
	\end{equation*}
	Thus by dominated convergence, we have 
	\begin{equation*}
		\theta_{2,\beta}(r) \defeq \int_{\tau_1}^{\tau_2} \int_I \int_{\{\abs{x}>r\}} \abs{\beta(t,u,x) }^2 dx du dt \to 0 
	\end{equation*}
	as $r \to \infty$. 

	\textit{Step 2.}
	We now look at the first part of~\eqref{e:telescope-beta-L2}. 
	Recalling that $\hat \beta^n_{h,\kappa} = \frac{\hat \pi^n_h}{\hat \mu^n_h \lor \kappa_2}$, we obtain
	\begin{align*}
		& \abs{\hat \beta^n_{h, \kappa} (t, u, x) - \beta (t, u, x)}^2 \lesssim \\
		& \qquad \kappa_2^{-2} \left( \abs{\hat \pi^n_h(t, u, x) - \pi(t, u, x)}^2 + \norm{b}_\infty^2 \abs{\hat \mu^n_h (t, u, x) - \mu(t, u, x)}^2 \right) 
	\end{align*}
	whenever $0 < \kappa_2 < \mu(t,u,x)$. 
	Note that $\mu$ has a strictly positive lower bound over $[\tau_1, \tau_2] \times I \times B(0,r)$ thanks to Harnack's inequality (see for instance Corollary 8.2.2 in~\cite{BogachevKrylovRocknerShaposhnikov2015}).
	This allows us to choose a strictly positive $\kappa_2$ depending on $r$. 
	We may set without loss of generality $\kappa_2 = \kappa_2(r)$ decreasing as $r$ increases. 

	We already have an upper bound of 
	\begin{equation*}
		\int_{\tau_1}^{\tau_2} \int_I \int_{\{\abs{x} \le r\}} \E \abs{\hat \mu^n_h (t, u, x) - \mu(t, u, x)}^2 dx du dt 
	\end{equation*}
	from Corollary~\ref{c:err-bound-mu-overall}.
	It remains to look at the errors of $\hat \pi$.  

	For the estimates of $\pi$, we rearrange and combine the terms in the upper bound given in Lemma~\ref{l:HJK-est-pi} to see that 
	\begin{align*}
		& \E \abs{\hat \pi^n_h(t, u, x) - \pi(t, u, x)}^2 \lesssim_{T,b,H,J,K} \\
		& \qquad n^{-1} h_1^{-1} h_2^{-2} h_3^{-2-2d} + n^{-1} h_1^{-2} h_2^{-2} h_3^{-2d} \\
		& \quad + n^{-2} h_1^{-1} h_3^{-2d} \sum_{i=1}^n J_{h_2}(u-\frac{i}{n})^2 \\
		& \quad + n^{-2} \int_0^T H_{h_1}^2(t - s) \sum_{i=1}^{n} J_{h_2}^2(u - \frac{i}{n}) \norm{K_{h_3}(x - \cdot)}_{L^2(\mu_{s, \frac{i}{n}})}^2 ds \\
		& \quad + n^{-3} h_2^{-2} \int_0^T H_{h_1}^2(t - s) \sum_{i=1}^{n} \norm{K_{h_3}(x - \cdot)}_{L^2(\mu_{s, \frac{i}{n}})}^2 ds \\
		& \quad + \abs{\HJK_h \ast \pi (t, u, x) - \pi(t, u, x)}^2 \,.
	\end{align*}
	Analogous to the proof of Corollary~\ref{c:err-bound-mu-overall}, integrating those items over $[\tau_1, \tau_2] \times I \times \{\abs{x}\le r\}$ produces
	\begin{align*}
		& \int_{\tau_1}^{\tau_2} \int_I \int_{\{\abs{x} \le r\}} \E \abs{\hat \pi^n_h(t, u, x) - \pi(t, u, x)}^2 \lesssim_{T,b,H,J,K} \\
		& \qquad r^d (n^{-1} h_1^{-1} h_2^{-2} h_3^{-2-2d} + n^{-1} h_1^{-2} h_2^{-2} h_3^{-2d}) \\
		& \quad + n^{-1} h_1^{-1} h_2^{-1} h_3^{-d} + n^{-2} h_1^{-1} h_2^{-2} h_3^{-d} \\
		& \quad + \norm{\HJK_h \ast \pi - \pi}_{L^2([\tau_1, \tau_2] \times I \times \R^d)}^2 \,.
	\end{align*}
	Recall that $\pi = \mu \beta$, where $\abs{\beta} \le \norm{b}_\infty$.  
	Using the same idea attaining~\eqref{e:JK-conv-L2}, we see that
	\begin{equation*}
		\lim_{h_1,h_2,h_3 \to 0} \norm{\HJK_h \ast \pi - \pi}_{L^2([\tau_1, \tau_2] \times I \times \R^d)}^2 = 0 \,,
	\end{equation*}
	and we denote the convergence rate by $\theta_{3,\pi}(h)$. 

	Therefore, joining all the items, we obtain an overall upper bound 
	\begin{align*}
		& \int_{\tau_1}^{\tau_2} \int_{I} \int_{\R^d} \E \abs{\hat \beta^n_{h,\kappa,r} (t, u, x) - \beta (t, u, x)}^2 dx du dt \lesssim \\
		& \qquad  \kappa_2(r)^{-2} \paren{ n^{-1} h_1^{-1} h_2^{-1} h_3^{-d} + n^{-2} h_1^{-1} h_2^{-2} h_3^{-d} } \\
		& \quad + \kappa_2(r)^{-2} r^{d} (n^{-1} h_1^{-2} h_2^{-2} h_3^{-2d} + n^{-1} h_1^{-1} h_2^{-2} h_3^{-2-2d}) \\
		& \quad + \kappa_2(r)^{-2} (\theta_{3,\mu}(h) + \theta_{3,\pi}(h)) + \theta_{2, \beta}(r) \,,
	\end{align*}
	finishing the proof. 
\end{proof}

\section{Proofs for Section~\ref{s:minimax}}\label{s:proof-minimax}

The main improvement in the estimations in Lemmata~\ref{l:improve-est-mu} and~\ref{l:improve-est-beta} over Lemmata~\ref{l:JK-est-mu} and~\ref{l:HJK-est-pi} is the elimination of the step of inequality~\eqref{e:telescoping}. 
At a given point $(t_0,u_0,x_0)$, we are able to remove a heavy error term by simply sacrificing a constant multiple (depending on the point $(t_0,u_0,x_0)$). 
This requires a change-of-measure argument thanks to Girsanov's theorem, and the analysis of the constant multiple follows from Proposition 19 of~\cite{MaestraHoffmann2022}.

Recall that the finite-population system has the following dynamics
\begin{equation*}
	d X^n_i(t) = \frac{1}{n} \sum_{j=1}^n g^n_{ij} b(X^n_i(t), X^n_j(t)) dt + \sigma (X^n_i(t)) d B_{\frac{i}{n}}(t) \,,
\end{equation*}
for $i = 1, \dots, n$. 
We define 
\begin{equation*}
	\bar B^n_i(t) = \int_0^t (\sigma \sigma^T)^{-1/2} (X^n_i(s)) (dX^n_i(s) - \beta(s,\frac{i}{n},X^n_i(s))ds) 
\end{equation*}
for $i = 1, \dots, n$, and $t \in [0, T]$. 
Then 
\begin{equation*}
	d X^n_i(t) = \beta (t, \frac{i}{n}, X^n_i(t)) dt + \sigma(X^n_i(t)) d \bar B^n_i(t) \,, \qquad i = 1, \dots, n \,.
\end{equation*}

Let $\bar M^n$ be the process 
\begin{equation*}
	\bar M^n_t = \sum_{i=1}^n \int_0^t \left( \frac{1}{n} \sum_{j=1}^n g^n_{ij} b(X^n_i(s), X^n_j(s)) - \beta (s, \frac{i}{n}, X^n_i(s)) \right)^T (\sigma \sigma^T)^{-1/2} (X^n_i(s)) d \bar B^n_i(s) \,.
\end{equation*}
Define a new probability measure $\bar \P$ via 
\begin{equation*}
	\frac{d \bar \P}{d \P} = \exp \left( \bar M^n_T - \frac{1}{2} \ip{\bar M^n}_T \right) \,,
\end{equation*}
where $\ip{\cdot}$ denotes the quadratic variation. 
Observe that $\{\bar B^n_i \st i = 1, \dots, n \}$ are independent $\bar \P$-Brownian motions, and that $\bar M$ is a $\bar \P$-martingale. 
So $\{X^n_i \st i = 1, \dots, n \}$ are independent under $\bar \P$, and the $\bar \P$-law of $X^n_i$ coincides with $\P$-law of $X_{\frac{i}{n}}$, respectively for every $i = 1, \dots, n$. 
A slight modification on Proposition 19 of~\cite{MaestraHoffmann2022} gives the following relation.  
\begin{lemma}
	\label{l:change-of-measure}
	There exist constants $C,a > 0$ such that,
	for any $\f_T$-measurable event $E$, we have
	\begin{equation*}
		\P (E) \le C (\bar \P(E))^a \,.
	\end{equation*}
	Here $\f_T$ is the $\sigma$-algebra generated by the Brownian motions $\{B_u(t)\}_{t \in [0,T], u \in I}$.
\end{lemma}

Now we have the tools to complete the proof of the improved estimations and thus the minimax analysis.

\subsection{Proof of Theorem~\ref{t:minimax-mu}}\label{s:minimax-mu-prf}

We first justify the improved upper bound. 

\begin{proof}[Proof of Lemma~\ref{l:improve-est-mu}]
	Observe that 
	\begin{align*}
		& \E \abs{ \hat \mu^n_h (t_0,u_0,x_0) - \mu(t_0,u_0,x_0) }^2  \\
		& \le 3 \E \abs{ \frac{1}{n} \sum_{i=1}^n J_{h_2}(u_0-\frac{i}{n}) \paren[\big]{ K_{h_3}(x_0 - X^n_i(t_0)) - \bar \E (K_{h_3}(x_0 - X^n_i(t_0))) } }^2 \\
		& + 3 \E \abs{ \int_I J_{h_2}(u_0-\frac{\ceil{nu}}{n}) \bar \E (K_{h_3}(x_0 - X^n_i(t_0))) - J_{h_2}(u_0-u) \E (K_{h_3}(x_0 - X_u(t_0))) du }^2 \\
		& + 3 \abs{ \JK_h \ast \mu_{t_0} (u_0,x_0) - \mu(t_0,u_0,x_0) } \\
		& =: 3 (M'_1 + M'_2 + M'_3) \,.
	\end{align*}
	For $i = 1, \dots, n$, let
	\begin{equation*}
		\bar Z_i = J_{h_2}(u_0-\frac{i}{n}) \paren[\big]{ K_{h_3}(x_0 - X^n_i(t_0)) - \bar \E (K_{h_3}(x_0 - X^n_i(t_0))) } \,.
	\end{equation*}
	Note that $\bar Z_i = 0$ whenever $\abs{u_0 - \frac{i}{n}} > h_2$, so the number of nonzero terms in the summation is $O(nh_2)$. 
	
	The main improvement upon Lemma~\ref{l:JK-est-mu} comes from the upper bound of $M'_1$ via the change-of-measure argument. 
	Following the same strategy as in the proof of Lemma~\ref{l:JK-est-mu}, we have 
	\begin{align*}
		M'_1 & = \E \abs{ \frac{1}{n} \sum_{i=1}^n \bar Z_i }^2 \\
		& = \int_0^\infty \P \left( \abs{ \sum_{i=1}^n \bar Z_i } > n \sqrt{z} \right) dx \\
		& \le C \int_0^\infty \bar \P \left( \abs{ \sum_{i=1}^n \bar Z_i } > n \sqrt{z} \right)^a dz \,.
	\end{align*}
	Recall that $\{X^n_i \st i = 1, \dots, n\}$ are independent under $\bar \P$. 
	Then so are $\{\bar Z_i \st i = 1, \dots, n\}$. 
	Moreover, we have $\bar \E \bar Z_i = 0$ and $\abs{\bar Z_i} \le 2 \norm{\JK_h}_\infty$ a.s. 
	We may thus apply Bernstein's inequality, 
	\begin{equation*}
		\bar \P \left( \abs{ \sum_{i=1}^n \bar Z_i } > n \sqrt{z} \right) \le 2 \exp \left( - \frac{\frac{1}{2} n^2 z}{ \sum_{i=1}^n \E \bar Z_i^2 + \frac{1}{3} n \sqrt{z} \norm{\JK_h}_\infty  } \right) \,.
	\end{equation*}
	For index $i$ such that $\abs{u_0 - \frac{i}{n}} \le h_2$, we have 
	\begin{align*}
		\E Z_i^2 & \le J_{h_2}(u_0-\frac{i}{n})^2 \bar \E (K_{h_3}(x_0-X^n_i(t_0))^2) \\ 
		& = J_{h_2}(u_0-\frac{i}{n})^2 \E (K_{h_3}(x_0-X_{\frac{i}{n}}(t_0))^2) \\ 
		& \le h_2^{-2} \norm{J}_\infty^2 \int_{\R^d} K_{h_3} (x_0 - x)^2 \mu_{t_0,\frac{i}{n}}(dx) \\
		& \le C_\mu h_2^{-2} h_3^{-d} \norm{J}_\infty^2 \norm{K}_2^2 \,,
	\end{align*}
	thanks to the local boundedness of $\mu_{t_0,\frac{i}{n}}$, in a neighborhood of $x_0$. 
	Using estimate (48) in~\cite{MaestraHoffmann2022}, we get 
	\begin{align*}
		& \int_0^\infty \bar \P \left( \abs{ \sum_{i=1}^n \bar Z_i } > n \sqrt{z} \right)^a dz \\
		& \quad \le 2 \max \left\{ \frac{2 C_\mu n h_2^{-1} h_3^{-d} \norm{J}_\infty^2 \norm{K}_2^2}{a n^2}, \paren[\big]{ \frac{2n h_2^{-1} h_3^{-d} \norm{\JK}_\infty}{3 a n^2} }^2 \right\} \\
		& \quad \le C (n^{-1} h_2^{-1} h_3^{-d} \norm{J}_\infty^2 \norm{K}_2^2 + n^{-2} h_2^{-2} h_3^{-2d} \norm{J}_\infty^2 \norm{K}_\infty^2 ) \,.
	\end{align*}

	The term $M'_2$ is where items (1) and (2) of Lemma~\ref{l:improve-est-mu} become different. 
	We first work on item (1).
	Recall that $J_{h_2}$ is supported on $\overline{B(0,h_2)}$.
	Then Cauchy-Schwarz inequality gives
	\begin{align*} 
		& M'_2 \le  \\ 
		& \; h_2 \int_I \abs{ J_{h_2}(u_0-\frac{\ceil{nu}}{n}) \bar \E (K_{h_3}(x_0 - X^n_{\ceil{nu}}(t_0))) - J_{h_2}(u_0-u) \E (K_{h_3}(x_0 - X_u(t_0))) }^2 du \,.
	\end{align*} 
	For each $u \in [u_0-h_2,u_0+h_2]$, since the $\bar \P$-law of $X^n_{\ceil{nu}}$ is identical to the $\P$-law of $X_{\frac{\ceil{nu}}{n}}$, we have 
	\begin{align*}
		& J_{h_2}(u_0-\frac{\ceil{nu}}{n}) \bar \E (K_{h_3}(x_0 - X^n_{\ceil{nu}}(t_0))) - J_{h_2}(u_0-u) \E (K_{h_3}(x_0 - X_u(t_0))) \\
		& = J_{h_2}(u_0-\frac{\ceil{nu}}{n}) \left( \E (K_{h_3}(x_0 - X_{\frac{\ceil{nu}}{n}}(t_0))) - \E (K_{h_3}(x_0 - X_u(t_0))) \right) \\
		& + \left( J_{h_2}(u_0 - \frac{\ceil{nu}}{n}) - J_{h_2}(u_0-u) \right) \E (K_{h_3}(x_0 - X_u(t_0))) \,.
	\end{align*}
	So 
	\begin{align*}
		& \abs{ J_{h_2}(u_0-\frac{\ceil{nu}}{n}) \bar \E (K_{h_3}(x_0 - X^n_{\ceil{nu}}(t_0))) - J_{h_2}(u_0-u) \E (K_{h_3}(x_0 - X_u(t_0))) }^2  \\
		& \quad \le 2 J_{h_2}(u_0-\frac{\ceil{nu}}{n})^2 \E \abs{ K_{h_3}(x_0 - X_{\frac{\ceil{nu}}{n}}(t_0)) - K_{h_3}(x_0 - X_{u}(t_0)) }^2 \\
		& \qquad + 2 \abs{J_{h_2}(u_0 - \frac{\ceil{nu}}{n}) - J_{h_2}(u_0-u)}^2 \E (K_{h_3}(x_0 - X_u(t_0))^2) \\
		& \quad \le 2 h_2^{-2} \norm{J}_\infty^2 h_3^{-2-2d} \norm{K}_\infty^2 \E \abs{ X_{\frac{\ceil{nu}}{n}}(t_0) - X_u(t_0) }^2 \\
		& \qquad + 2 h_2^{-4} n^{-2} \norm{\grad J}_\infty^2 C_\mu h_3^{-d} \norm{K}_2^2 \\
		& \quad \le C (n^{-2} h_2^{-2} h_3^{-2-2d} + n^{-2} h_2^{-4} h_3^{-d}) \,,
	\end{align*}
	where the last inequality uses Theorem 2.1 of~\cite{BayraktarChakrabortyWu2023}. 

	Integrating the above errors, we obtain 
	\begin{equation*}
		M'_2 \le C n^{-2} ( h_3^{-2-2d} \norm{J}_\infty^2 \norm{\grad K}_\infty^2 + h_2^{-2} h_3^{-d} \norm{\grad J}_\infty^2 \norm{K}_2^2 ) \,.
	\end{equation*}
	That finishes the proof of item (1). 

	Looking at the proof of item (1), we notice that the only difference in item (2) compared to item (1) happens at the term 
	\begin{equation*}
		\E \abs{ K_{h_3}(x_0 - X_{\frac{\ceil{nu}}{n}}(t_0)) - K_{h_3}(x_0 - X_u(t_0)) }^2 \,.
	\end{equation*}
	The previous (crude) analysis (in the proof of Lemma~\ref{l:JK-est-mu}) gives an upper bound $O(n^{-2} h_3^{-2-2d})$ simply by mean-value theorem. 
	However, the use of mean-value theorem is activated only when $\abs{x_0 - X_{\frac{\ceil{nu}}{n}}(t_0)} \le h_3$ and $\abs{x_0 - X_{u}(t_0)} \le h_3$. 
	Given the local boundedness of $\mu$, we have 
	\begin{equation*}
		\P(A(t_0,u,x_0)) \defeq \P \left( \abs{x_0 - X_{\frac{\ceil{nu}}{n}}(t_0)} \le h_3 \text{ or } \abs{x_0 - X_{u}(t_0)} \le h_3 \right) \le 2 C_\mu' h_3^d
	\end{equation*}
	for some constant $C_\mu'$. 

	Now, with the additional assumption on the continuity of initial data with respect to the $p$-Wasserstein metric, we adjust the proof of Theorem 2.1(b) in~\cite{BayraktarChakrabortyWu2023} to see that 
	\begin{equation*}
		\E \abs{ X_{\frac{\ceil{nu}}{n}}(t_0) - X_u(t_0) }^p \le C_p' n^{-p} \,.
	\end{equation*}
	Then, by Hölder's inequality, we get 
	\begin{align*}
		& \E \abs{ K_{h_3}(x_0 - X_{\frac{\ceil{nu}}{n}}(t_0)) - K_{h_3}(x_0 - X_u(t_0)) }^2 \\
		& \le \E \left( \norm{\grad K_{h_3}}_\infty^2 \abs{ X_{\frac{\ceil{nu}}{n}}(t_0) - X_u(t_0) }^2 \one{A(t_0,u,x_0)} \right) \\
		& \le h_3^{-2-2d} \norm{\grad K}_\infty^2 \left( \abs{ X_{\frac{\ceil{nu}}{n}}(t_0) - X_u(t_0) }^p \right)^{\frac{2}{p}} \P(A(t_0,u_0,x_0))^{\frac{p-2}{p}} \\
		& \le C_p' n^{-p} h_3^{-2-\frac{p+2}{p}d} \norm{\grad K}_\infty^2 \,.
	\end{align*}
	Note that $C_p'$ is independent of $u$. 
	That finishes the proof of item (2). 
\end{proof}

It remains to analyze the bias term.
Fix $t_0 \in (0, T)$, $u_0 \in (0, 1)$, and $x_0 \in \R^d$. 
When $h_2 < u_0$, we have 
\begin{align*}
	& \JK_h \ast \mu(t_0, u_0, x_0) - \mu(t_0, u_0, x_0) = \\
	& \qquad  \int_\R \int_{\R^d} J_{h_2}(u_0-u) K_{h_3}(x_0-x) (\mu(t_0,u,x) - \mu(t_0,u_0,x_0)) dxdu \,.
\end{align*}
For $u \in I$ and $x \in \R^d$ such that $\abs{u_0 - u} < h_2$ and $\abs{x_0 - x} < h_3$, we have 
\begin{align*}
	& \mu(t_0,u,x) - \mu(t_0,u_0,x_0) = \\
	& \qquad (\mu(t_0,u,x) - \mu(t_0,u_0,x)) + (\mu(t_0,u_0,x) - \mu(t_0,u_0,x_0)) \,.
\end{align*}
The second term has order $O(\abs{h_3}^{s})$ due to the selection of Hölder continuity class. 
We will bound the first term with the following technical lemma. 

\begin{lemma}
	\label{l:lip-cont-uv}
	Assume the hypothesis of 
	There exists some constant $C_I > 0$, depending only on $T,d,b,\sigma$, such that 
	\begin{equation*}
		\abs{ \mu(t,u,x) - \mu(t,v,x) } \le C_I \abs{u-v} \qquad dx\text{-a.s.}
	\end{equation*}
	for every $u,v \in I$ and every $t \in [0,T]$. 
\end{lemma}
The proof consists of several properties of parabolic equations, and we defer it to the Appendix~\ref{s:pf-technical}.

With the technical estimates given above, we are now able to prove Theorem~\ref{t:minimax-mu}. 
We start with the upper bound. 

\begin{proof}[Proof of Theorem~\ref{t:minimax-mu}, upper bound]
	We first work under assumption (a). 

	Given $(b,\sigma,G,\mu_0) \in \a^{s}_L(t_0,x_0)$, we know that 
	\begin{equation*}
		\abs{\mu(t_0,u_0,x) - \mu(t_0,u_0,x_0)} \le L \abs{x - x_0}^{s} \le L h_3^s 
	\end{equation*}
	whenever $x_0 - x \in supp(K_{h_3})$. 
	Thanks to Lemma~\ref{l:lip-cont-uv}, we have 
	\begin{equation*}
		\abs{\mu(t_0,u,x) - \mu(t_0,u_0,x)} \le C_I \abs{u-u_0} \le C_I h_2 \,,
	\end{equation*}
	whenever $u, u_0 \in I$. 
	So the bias term is bounded by 
	\begin{equation*}
		\abs{\JK_h \ast \mu_{t_0}(u_0,x_0) - \mu(t_0,u_0,x_0)}^2 \le 2 (C_I^2 + L^2) (h_2^2 + h_3^{2s}) \,.
	\end{equation*}
	Then, along with Lemma~\ref{l:improve-est-mu}, the total upper bound of the estimation error is given by 
	\begin{equation*}
		\E \abs{\hat \mu^n_h - \mu}^2 \le C (n^{-1} h_2^{-1} h_3^{-d} + n^{-2} h_2^{-2} h_3^{-2d} + n^{-2} h_3^{-2-2d} + n^{-2} h_2^{-2} h_3^{-d} + h_2^2 + h_3^{2s} ) \,.
	\end{equation*}
	Taking $h_2 = n^{-\frac{s}{d+3s}}$ and $h_3 = n^{-\frac{1}{d+3s}}$, we get 
	\begin{equation*}
		\E \abs{ \hat \mu^n_h (t_0,u_0,x_0) - \mu(t_0,u_0,x_0) }^2 \lesssim n^{-\frac{2s}{d+3s}} + n^{-\frac{6s-2}{d+3s}} \lesssim n^{-\frac{2s}{d+3s}} \,.
	\end{equation*}
	The last inequality holds when $s \ge \frac{1}{2}$. 
	Note that the implicit constant in the inequality depends only $T,d,C_I,L,\norm{\JK}_2,\norm{\JK}_\infty$, and the values of $\mu$ near $(t_0,u_0,x_0)$. 

	Next, we work under the assumption (b). 
	Analogous to above, we have 
	\begin{equation*}
		\E \abs{\hat \mu^n_h - \mu}^2 \le C (n^{-1} h_2^{-1} h_3^{-d} + n^{-2} h_2^{-2} h_3^{-2d} + n^{-2} h_3^{-2-\frac{p+2}{p}d} + n^{-2} h_2^{-2} h_3^{-d} + h_2^2 + h_3^{2s} ) \,.
	\end{equation*}
	Taking $h_2 = n^{-\frac{s}{d+3s}}$ and $h_3 = n^{-\frac{1}{d+3s}}$, we get 
	\begin{equation*}
		\E \abs{ \hat \mu^n_h (t_0,u_0,x_0) - \mu(t_0,u_0,x_0) }^2 \lesssim n^{-\frac{2s}{d+3s}} + n^{-\frac{6s-2+d-2d/p}{d+3s}} \,.
	\end{equation*}
	As $0 < s < \frac{1}{2}$, $p > 2$, and $p (2-4s) \le (p-2) d$, we have 
	\begin{equation*}
		6s-2 + d-\frac{2d}{p} \ge 2s \,.
	\end{equation*}
	This leads to the final asymptotic upper bound in~\eqref{e:minimax-mu-upper}, namely
	\begin{equation*}
		n^{-\frac{2s}{d+3s}} \,.
	\end{equation*}
\end{proof}

Finally, we demonstrate the lower bound using Le Cam's two-point method (see for instance Chapter 2 of \cite{LeCam1986}). 
We shall construct two examples of graphon mean-field systems such that, the total variation distance between their laws is bounded by $\frac{1}{2}$, while the densities at $(t_0,u_0,x_0)$ differ by some quantity of order $n^{-\frac{2s}{d+3s}}$. 
The construction is adapted from~\cite{MaestraHoffmann2022}, with an extra factor for the graphon index $u \in I$, so we will skip some technical details in the proof below.   

\begin{proof}[Proof of Theorem~\ref{t:minimax-mu}, lower bound]
	\textit{Step 1.}
	We consider graphon particle systems with no interactions. 

	Pick a smooth potential function $U_1: \R^d \to \R$ such that $\grad U_1$ is Lipschitz, $U_1 = 0$ in a neighborhood of $x_0$, and 
	\begin{equation*}
		\limsup_{\abs{x} \to \infty} \frac{x^T \grad U_1 (x)}{\abs{x}^2} > 0 \,.
	\end{equation*}
	Define the drift $b(x,y) = b_1(x) \defeq -\grad U_1(x)$.
	Pick a Lipschitz continuous function $G_1: I \times \R$ such that $G_1 = 0$ in a neighborhood of $u_0$, and define the graphon weight $G(u,v) = G_1(u)$. 
	We set $C_{1,u} \defeq \int_{\R^d} \exp (-2 G_1(u) U_1(x)) dx$ and define 
	\begin{equation*}
		\nu_1(u,x) = C_{1,u}^{-1} \exp (-2 G_1(u) U_1(x)) \,, \qquad u \in I \,.
	\end{equation*}
	Then we obtain a family of diffusion processes $\{X_u\}_{u \in I}$ such that 
	\begin{equation}
		\label{e:inv-sys-1}
		d X_u(t) = b_1(X_u(t)) G_1(u) dt + d B_u(t) \,, \quad X_u(0) \sim \nu_1(u)\,, \qquad u \in I \,.
	\end{equation}
	Notice that $X_u$'s are independent, and $\nu_1(u)$ is the invariant distribution of $X_u$. 
	This gives a graphon particle system with time-invariant density function $\nu_1$. 
	In particular, we may assume that $(b_1, I_{d \times d}, G_1, \nu_1) \in S^{s}_{L/2}(t_0,x_0)$. 
	
	Now we consider a deviation from the system~\eqref{e:inv-sys-1}.
	Let $\psi \in C_c^\infty(\R \times \R)$ be a cut-off function such that 
	\begin{itemize}
		\item $\psi(0,0) = 1$ and $\norm{\psi}_\infty = 1$, 
		\item $\int_{\R^d} \psi(u,x) dx = 0$ for every $u \in \R$, and $\norm{\psi}_2 = 1$, 
		\item $\sup_{u \in \R} \norm{\psi(u,\cdot)}_{\h^s(x_0)} < \infty$. 
	\end{itemize}
	Let $\alpha \in (0,1)$ be sufficiently small. 
	Then define $U_2: \R^d \to \R$ and $G_2: I \to [0,1]$ so that
	\begin{equation*}
		G_2^n(u) U_2^n (x) = G_1(u) U_1(x) + \alpha C_{1,u} n^{-1/2} \zeta_n^{1/2} \tau_n^{d/2} \psi( \zeta_n(u-u_0), \tau_n(x-x_0) ) \,,
	\end{equation*}
	where $\tau_n, \zeta_n$ are positive scalars that tend to $\infty$ as $n \to \infty$. 
	Let $b_2^n = - \grad U_2^n$. 
	Then we construct the second particle system similar to the above, with time-invariant density 
	\begin{equation*}
		\nu_2^n(u,x) = C_{2,n,u}^{-1} \exp (-2 G_2^n(u) U_2^n(x)) \,, \quad C_{2,n,u} \defeq \int_{\R^d} \exp (-2 G_2^n(u) U_2^n(x)) dx \,.
	\end{equation*}
	Moreover, to maintain the desired Lipschitz and Hölder continuity, we need 
	\begin{equation*}
		n^{-1/2} \zeta_n^{3/2} \tau_n^{s+d/2} \lesssim 1 \,.
	\end{equation*} 
	This allows us to take $\tau_n = n^{\frac{1}{d+3s}}$ and $\zeta_n = n^{\frac{s}{d+3s}}$, which also ensures that $(b_2^n, I_{d \times d}, G_2^n, \nu_2^n) \in S^s_L(t_0,x_0)$. 
	
	\textit{Step 2.}
	Now we run the finite-population systems derived from the above two graphon particle systems and make observations of the particle positions. 
	For~\eqref{e:inv-sys-1}, the $n$ particles display the dynamics 
	\begin{equation*}
		d X^n_i(t) = b_1(X^n_i(t)) G(\frac{i}{n}) dt + d B_{\frac{i}{n}}(t) \,, \qquad i = 1, \dots, n \,.
	\end{equation*}
	The distributions of the particles coincide with those in the graphon system, with joint law 
	\begin{equation*}
		\mu_1 \defeq \bigotimes_{i=1}^n \nu_1(\frac{i}{n}) \,.
	\end{equation*}
	Similarly, the joint law in the second system is given by 
	\begin{equation*}
		\mu_2^n \defeq \bigotimes_{i=1}^n \nu_2^n(\frac{i}{n}) \,.
	\end{equation*}
	Then, following the strategy in~\cite{MaestraHoffmann2022}, with Pinsker's inequality, we have  
	\begin{align*}
		\norm{\mu_1 - \mu_2^n}_{TV}^2 \le 2 \sum_{i=1}^n \abs{\log \frac{C_{2,n,\frac{i}{n}}}{C_{1,\frac{i}{n}}}}
	\end{align*}

	Taylor's theorem gives 
	\begin{align}
		\nonumber
		\abs{\log \frac{C_{2,n,\frac{i}{n}}}{C_{1,\frac{i}{n}}}} & \le \abs{ \frac{C_{2,n,\frac{i}{n}}}{C_{1,\frac{i}{n}}} - 1 } \\
		\label{e:taylor-nu1nu2}
		& = 2 \alpha^2 n^{-1} \zeta_n \tau_n^{d} \int_{\R^d} \nu_1(u,x)^{-1} \psi (\zeta_n(\frac{i}{n}-u_0), \tau_n(x-x_0))^2 R_i(x) dx \,,
	\end{align}
	where the remainder term $R_i \in [0,2]$. 
	Notice that $\nu_1(u,x)^{-1}$ is bounded above in a neighborhood of $(u_0,x_0)$. 
	So there exists some constant $c_1$ such that 
	\begin{equation*}
		\norm{\mu_1 - \mu_2^n}_{TV}^2 \le \frac{c_1 \alpha^2 \zeta_n}{n} \sum_{i=1}^n \norm{\psi (\zeta_n(\frac{i}{n}-u_0), \cdot)}_2^2 \,.
	\end{equation*}
	Since $\psi \in C_c^\infty$, we have
	\begin{align*}
		& \int_I \int_{\R^d} \abs{ \psi ( \zeta_n(\frac{\ceil{nu}}{n}-u_0), x)^2 - \psi (\zeta_n(u-u_0), x)^2 } dx du \\
		& \le 2 \int_I \int_{\R^d} \abs{\zeta_n (\frac{\ceil{nu}}{n} - u)} \psi_0(x) dx du \\
		& \le 2 n^{-1} \zeta_n \norm{\psi_0}_1 \,,
	\end{align*}
	where we set $\psi_0 \defeq \sum_{k=1}^d \norm{\grad_u \psi_k (x)}_{L^\infty(\R)}$. 
	This implies 
	\begin{equation*}
		\frac{1}{n} \sum_{i=1}^n \norm{\psi (\zeta_n(\frac{i}{n}-u_0), \cdot)}_2^2 \le \zeta_n^{-1} \norm{\psi}_2^2 + 2 n^{-1} \zeta_n \norm{\psi_0}_1 \,.
	\end{equation*}
	Thus 
	\begin{equation*}
		\norm{\mu_1 - \mu_2^n}_{TV}^2 \le c_1 \alpha^2 + o(1) \le \frac{1}{4}
	\end{equation*}
	when $\alpha$ is chosen to be small enough and $n$ is sufficiently large. 

	\textit{Step 3.}
	Finally, we apply the Le Cam's lemma to see that 
	\begin{align*}
		& \inf_{\hat \mu} \sup_{ (b,\sigma,G,\mu_0) \in \a^{s}_L (t_0,x_0) } \E \abs{ \hat \mu - \mu(t_0,u_0,x_0) } \\
		\ge & \inf_{\hat \mu} \max_{\tilde \mu \in \{\mu_1, \mu_2^n\}} \E \abs{\hat \mu - \tilde \mu (t_0,u_0,x_0)} \\
		\ge & \frac{1}{2} \abs{\nu_1(u_0,x_0) - \nu_2^n(u_0,x_0)} (1 - \norm{\mu_1 - \mu_2^n}_{TV}) \\
		\ge & \frac{1}{4} \abs{\nu_1(u_0,x_0) - \nu_2^n(u_0,x_0)} \,.
	\end{align*}
	The same strategy as in~\eqref{e:taylor-nu1nu2} (see also equation (68) in~\cite{MaestraHoffmann2022}) gives that 
	\begin{equation*}
		\abs{\nu_1(u_0,x_0) - \nu_2^n(u_0,x_0)} \gtrsim n^{-1/2} \zeta_n^{1/2} \tau_n^{d/2} = n^{-\frac{s}{d+3s}} \,.
	\end{equation*}
	Therefore, we get the lower bound~\eqref{e:minimax-mu-lower} as well:
	\begin{equation*}
		\inf_{\hat \mu} \sup_{ (b,\sigma,G,\mu_0) \in \a^{s}_L (t_0,x_0) } \E \abs{ \hat \mu - \mu(t_0,u_0,x_0) }^2 \ge n^{-\frac{2s}{d+3s}} \,,
	\end{equation*}
	completing the proof of Theorem~\ref{t:minimax-mu}.
\end{proof}

\subsection{Proof of Theorem~\ref{t:minimax-beta}}\label{s:minimax-beta-prf}
\restartsteps

We first justify the improved upper bound. 
For notation simplicity, let us define the $n,p$-norm of a function $f: [0,T] \times I \times \R^d \to \R^d$ via 
\begin{equation*}
	\norm{\varphi}_{n,p}^p \defeq \int_0^T \frac{1}{n} \sum_{i=1}^n \int_{\R^d} f(t, \frac{i}{n}, x)^p dx dt \,.
\end{equation*}

\begin{proof}[Proof of Lemma~\ref{l:improve-est-beta}]
	Fix $t_0, u_0, x_0$. 
	We consider the following telescoping sum
	\begin{align*}
		& \hat \pi^n_h (t_0, u_0, x_0) - \pi(t_0, u_0, x_0) \\
		& = \int_0^T \frac{1}{n} \sum_{i=1}^n \HJK_h (t_0-t, u_0-\frac{i}{n}, x_0-X^n_i(t)) (Y^n_i(t) - \beta(t,\frac{i}{n}, X^n_i(t))) dt \\
		& + \int_0^T \int_I \int_{\R^d} \HJK_h(t_0-t, u_0-u, x_0-x) \beta(t,u,x) (\mu^n_t(du,dx) - \mu_{t,u}(dx)du) dt \\
		& + \int_0^T \frac{1}{n} \sum_{i=1}^n \HJK_h (t_0-t, u_0-\frac{i}{n}, x_0-X^n_i(t)) \sigma(X^n_i(t)) dB_{\frac{i}{n}}(t) \\
		& + (\HJK_h \ast \pi - \pi) (t_0,u_0,x_0) \,.
	\end{align*}
	This allows us to write 
	\begin{equation*}
		\E \abs{ \hat \pi^n_h (t_0, u_0, x_0) - \pi(t_0, u_0, x_0) }^2 \le 4 ( \E \abs{P'_1}^2 + \E \abs{P'_2}^2 + \E \abs{P'_3}^2 + \E \abs{P'_4}^2 )
	\end{equation*}
	with obvious definitions of the four components. 

	\textit{Step 1.}
	To bound $P'_2$, we observe that 
	\begin{align*}
		P'_2 & = \int_{[0,T] \times I \times \R^d} (\varphi \beta) (t,u,x) ( \mu^n_t(du,dx) - \mu_{t,u}(dx)du) dt \\
		& = \int_{[0,T] \times I \times \R^d} (\varphi \beta) (t,u,x) ( \mu^n_t(du,dx) - \bar \mu^n_t(du,dx)) dt \\
		& \quad + \int_{[0,T] \times I \times \R^d} (\varphi \beta) (t,u,x) (\bar \mu^n_t(du,dx) - \mu_{t,u}(dx)du) dt \,,
	\end{align*}
	where $\bar \mu^n_t(du,dx) = \frac{1}{n} \sum_{i=1}^n \mu_{t,u}(dx) \delta_{\frac{i}{n}}(du)$, and $\varphi = \HJK_h(t_0-\cdot, u_0-\cdot, x_0-\cdot)$. 
	Then the first part becomes 
	\begin{equation*}
		P'_{2,1} \defeq \frac{1}{n} \int_0^T (\varphi \beta) (t,\frac{i}{n},X^n_i(t)) - \bar \E [(\varphi \beta) (t,\frac{i}{n},X^n_i(t))] dt \,,
	\end{equation*} 
	where $\bar \P$ is the measure under which $(X^n_i)_{i=1,\dots,n}$ are independent, with corresponding laws $\mu_{\frac{i}{n}}$, $i=1,\dots,n$. 
	Since $\beta$ is bounded, and $\mu(t,u,x)$ is bounded in a small neighborhood of $(t_0,u_0,x_0)$, we have 
	\begin{equation*}
		\bar \E [(\varphi \beta) (t,\frac{i}{n},X^n_i(t))]^2 \lesssim H_{h_1}^2(t_0-t) J_{h_2}^2(u_0-\frac{i}{n}) \norm{K_{h_3}}_2^2 \,,
	\end{equation*}
	where the implicit constant is uniform over $i = 1, \dots, n$. 
	This gives 
	\begin{equation*}
		\int_0^T \sum_{i=1}^n \bar \E [(\varphi \beta) (t,\frac{i}{n},X^n_i(t))]^2 \lesssim \norm{\HJK_h}_{n,2}^2 \,.
	\end{equation*}
	Applying Bernstein's inequality gives 
	\begin{align*}
		\P \paren[\big]{ \abs{P'_{2,1}} \ge \sqrt{z} } & \le c_1 \bar \P \paren[\big]{ \abs{P'_{2,1}} \ge \sqrt{z} }^{c_2} \\
		& \le 2 d c_1 \exp \left( -\frac{c_2 n z}{c_3 (T^{-1} \norm{\HJK_h}_{n,2}^2 + \sqrt{z} \norm{\HJK_h}_\infty)} \right) \,,
	\end{align*}
	for some positive constants $c_1, c_2, c_3$. 
	Note that $J_{h_2}$ is supported on $[u_0-h_2, u_0+h_2]$, so 
	\begin{align*}
		\norm{\HJK_h}_{n,2}^2 & = \int_0^T H_{h_1}^2(t_0-t) \frac{1}{n} \sum_{i=1}^n J_{h_2}^2(u_0-\frac{i}{n}) \norm{K_{h_3}}_2^2 dt \\
		& \lesssim \norm{H_{h_1}}_2^2 h_2 \norm{J_{h_2}}_\infty^2 \norm{K_{h_3}}_2^2 \\
		& \lesssim h_1^{-1} h_2^{-1} h_3^{-d} \,.
	\end{align*}
	Thus 
	\begin{align*}
		\E \abs{P'_{2,1}}^2 & \lesssim 2 d c_1 \int_0^\infty \exp \left( -\frac{c_2 n z}{c_3 (T^{-1} \norm{\HJK_h}_{n,2}^2 + \sqrt{z} \norm{\HJK_h}_\infty)} \right) dz \\
		& \lesssim n^{-1} h_1^{-1} h_2^{-1} h_3^{-d} + n^{-2} h_1^{-2} h_2^{-2} h_3^{-2d} \,.
	\end{align*}

	For the second part, we observe that 
	\begin{align*}
		P'_{2,2} & \defeq \int_{[0,T] \times I \times \R^d} (\varphi \beta) (t,u,x) (\bar \mu^n_t(du,dx) - \mu_{t,u}(dx)du) dt \\
		& = \int_0^T \int_I \E \left[ (\varphi \beta) (t,\frac{\ceil{nu}}{n},X_{\frac{\ceil{nu}}{n}}(t)) - (\varphi \beta) (t,u,X_u(t)) \right] du dt \,.
	\end{align*}
	This is identical to $P_3$ in the proof of Lemma~\ref{l:HJK-est-pi}, which gives 
	\begin{equation*}
		\abs{P'_{2,2}} \lesssim n^{-2} h_1^{-1} h_3^{-2d} (h_2^{-2} + h_3^{-2}) \,.
	\end{equation*}
	Thus 
	\begin{equation*}
		\E \abs{P'_2} \lesssim n^{-1} h_1^{-1} h_2^{-1} h_3^{-d} + n^{-2} h_1^{-1} h_3^{-2d} (h_1^{-1} h_2^{-2} + h_3^{-2})
	\end{equation*}

	\textit{Step 2.}
	To bound $P'_1$, let us abuse the notation $\beta: [0,T] \times I \times \R^d \times \prob(I \times \R^d) \to \R^d$ for the mean-field dependent quantity 
	\begin{equation*}
		\beta(t,u,x,\mu_t) = \int_I \int_{\R^d} G(u,v) b(x,y) \mu_{t,u}(dx)dv \,.
	\end{equation*}
	Then we may write 
	\begin{align*}
		Y^n_i(t) - \beta(t,\frac{i}{n}, X^n_i(t)) & = \beta(t,\frac{i}{n},X^n_i(t),\mu^n_t) - \beta(t,\frac{i}{n},X^n_i(t),\mu_t) \\
		& = \int_I \int_{\R^d} G(\frac{i}{n},u) b(X^n_i(t),x) (\mu^n_t(du,dx) - \mu_{t,u}(dx) du) \,.
	\end{align*}
	Note that $b$ and $G$ are bounded and Lipschitz, we may compare with Proposition 19 under Assumption 4(iii)~\cite{MaestraHoffmann2022}.
	This gives a uniform-in-time bound 
	\begin{equation*}
		\P \left( \abs{Y^n_i(t) - \beta(t,\frac{i}{n}, X^n_i(t))} \ge z  \right) \le c_1 \exp \paren[\big]{ -\frac{c_2 n z^2}{1 + \sqrt{n} z} } \,, \qquad z > 0 \,.
	\end{equation*} 

	Now, applying Cauchy-Schwarz twice, we obtain 
	\begin{align*}
		\E \abs{P'_1}^2 & \le \left[ \E \left( \int_0^T \frac{1}{n} \sum_{i=1}^n \HJK_h(t_0-t,u_0-\frac{i}{n},x_0-X^n_i(t))^2  dt \right)^2 \right]^{1/2} \\
		& \quad \cdot \left[ \E \left( \int_0^T \frac{1}{n} \sum_{i=1}^n \abs{Y^n_i(t) - \beta(t,\frac{i}{n}, X^n_i(t))}^2 dt \right)^2 \right]^{1/2}
	\end{align*}
	Applying Cauchy-Schwarz to the second term again, we see that 
	\begin{align*}
		& \left[ \E \left( \int_0^T \frac{1}{n} \sum_{i=1}^n \abs{Y^n_i(t) - \beta(t,\frac{i}{n}, X^n_i(t))}^2 dt \right)^2 \right]^{1/2} \\
		& \le T^{1/2} \left( \int_0^T \frac{1}{n} \sum_{i=1}^n \E  \abs{Y^n_i(t) - \beta(t,\frac{i}{n}, X^n_i(t))}^4 dt \right)^{1/2} \\
		& \lesssim \sup_{t \in [0,T]} \sup_{i = 1, \dots, n} \left( \int_0^\infty \P \left( \abs{Y^n_i(t) - \beta(t,\frac{i}{n}, X^n_i(t))} \ge z^{1/4}  \right) dz \right)^{1/2} \\
		& \lesssim \int_0^\infty 2 c_1 \exp \paren[\big]{ -\frac{c_2 n z^{1/2}}{1 + \sqrt{n} z^{1/4}} }^{1/2}  \lesssim n^{-1} \,.
	\end{align*}
	For the first term, Minkowski's inequality gives an upper bound
	\begin{align}
		\nonumber
		& \int_{[0, T] \times I \times \R^d} \varphi(t,u,x)^2 \mu_{t,u}(dx)dudt \\
		\label{e:bernstein-4th-norm}
		& + \left[ \E \left( \int_{[0, T] \times I \times \R^d} \varphi(t,u,x)^2 (\mu^n_t(du,dx) - \mu_{t,u}(dx)du ) dt \right)^2 \right]^{1/2} \,,
	\end{align}
	where $\varphi = \HJK_h(t_0-\cdot, u_0-\cdot, x_0-\cdot)$. 
	With the same strategy applied to bound $P'_{2,1}$, we get 
	\begin{align*}
		& \E \left[ \int_{[0,T] \times I \times \R^d} \varphi(t,u,x)^2 (\mu^n_t(du,dx) - \bar \mu^n_t(du,dx)) dt \right]^2  \\ 
		& \qquad \qquad \lesssim n^{-1} h_1^{-3} h_2^{-3} h_3^{-3d} + n^{-2} h_1^{-4} h_2^{-4} h_3^{-4d} \,,
	\end{align*}
	which follows from the fact that 
	\begin{equation*}
		\norm{\varphi}_{n,4}^4 \lesssim h_1^{-3} h_2^{-3} h_3^{-3d} \,.
	\end{equation*}
	On the other hand, using the idea of $P'_{2,2}$, we get 
	\begin{align*}
		& \E \left[ \int_{[0,T] \times I \times \R^d} \varphi(t,u,x)^2 (\bar \mu^n_t(du,dx) - \bar \mu_{t,u}(dx)du) dt \right]  \\ 
		& \qquad \qquad \lesssim n^{-1} h_1^{-1} h_2^{-2} h_3^{-d} + n^{-1} h_1^{-1} h_2^{-2} h_3^{-1-2d} \,.
	\end{align*}
	Those lead to the following asymptotic upper bound for the first term, 
	\begin{align*}
		& h_1^{-1} h_2^{-1} h_3^{-d} + n^{-1/2} h_1^{-3/2} h_2^{-3/2} h_3^{-3d/2} + n^{-1} h_1^{-2} h_2^{-2} h_3^{-2d} \\
		& + n^{-1} h_1^{-1} h_2^{-2} h_3^{-d} + n^{-1} h_1^{-1} h_2^{-2} h_3^{-1-2d} \,,
	\end{align*}
	Joining the above leads to 
	\begin{equation*}
		\E \abs{P'_1}^2 \lesssim n^{-1} h_1^{-1} h_2^{-1} h_3^{-d} + n^{-2} h_1^{-1} h_2^{-2} h_3^{-1-2d} \,,
	\end{equation*}
	where we assume $n^{-1} h_1^{-1} h_2^{-1} h_3^{-d} \lesssim 1$. 

	\textit{Step 3.}
	Now, to bound $P'_3$, we apply It\^o's isometry to see that 
	\begin{align*}
		\E \abs{P'_3}^2 & \le \frac{\sigma_+^2 d^2}{n} \E \left[ \int_0^T \frac{1}{n} \sum_{i=1}^n \varphi(t,\frac{i}{n},X^n_i(t))^2 dt \right] \\
		& = \frac{\sigma_+^2 d^2}{n} \E \left[ \int_{[0,T] \times I \times \R^d} \varphi(t,u,x)^2 \mu^n_t(du,dx) dt \right] \,.
	\end{align*}
	Note that 
	\begin{equation*}
		\E \left[ \int_{[0,T] \times I \times \R^d} \varphi(t,u,x)^2  \mu_{t,u}(dx)du dt \right] \lesssim \norm{\varphi}_2^2 \,.
	\end{equation*}
	We write again 
	\begin{equation*}
		\mu^n_t(du,dx) - \mu_{t,u}(dx)du = (\mu^n_t(du,dx) - \bar \mu^n_t(du,dx)) + (\bar \mu^n_t (du,dx) - \mu_{t,u}(dx)du) \,.
	\end{equation*}
	Using the bound for~\eqref{e:bernstein-4th-norm} in Step 2 leads to
	\begin{align*}
		\E \abs{P'_3}^2  \lesssim n^{-1} h_1^{-1} h_2^{-1} h_3^{-d} + n^{-2} h_1^{-1} h_2^{-2} h_3^{-1-2d} \,,
	\end{align*}
	whenever $n^{-1} h_1^{-1} h_2^{-1} h_3^{-d} \lesssim 1$.

	Summarizing the above, we conclude that
	\begin{equation*}
		\E \abs{\hat \pi^n_h - \pi}^2 \lesssim n^{-1} h_1^{-1} h_2^{-1} h_3^{-d} + n^{-2} h_1^{-1} h_2^{-2} h_3^{-1-2d} + n^{-2} h_1^{-1} h_3^{-2-2d} + \abs{\varphi \ast \pi - \pi}^2 \,.
	\end{equation*}
\end{proof}

Finally, analogous to the analysis of $\mu$, we are able to show the optimality of $\hat \beta$. 

\begin{proof}[Proof of Theorem~\ref{t:minimax-beta}]
	The strategies are identical to the proof of Theorem~\ref{t:minimax-mu}.

	\step[Upper bound]
	Fix $t_0 \in (0,T)$, $u_0 \in I$, $x_0 \in \R^d$, and $\kappa_2 > 0$. 
	We assume that $n^{-1} h_1^{-1} h_2^{-1} h_3^{-d} \lesssim 1$. 

	Recall that 
	\begin{equation*}
		\abs{\hat \beta_{h,\kappa} - \beta}^2 \lesssim \kappa_2^{-2} (\abs{\hat \pi^n_h - \pi}^2 + \norm{b}_\infty^2 \abs{\hat \mu^n_h - \mu}^2) 
	\end{equation*}
	whenever $\kappa_2 < \inf_{\supp \HJK_h} \mu$. 
	From Lemmata~\ref{l:improve-est-mu} and~\ref{l:improve-est-beta}, we have 
	\begin{align*}
		& \E \abs{\hat \mu^n_h - \mu}^2 \lesssim \\
		& \qquad n^{-1} h_2^{-1} h_3^{-d} + n^{-2} h_2^{-2} h_3^{-2d} + n^{-2} h_3^{-2-2d} + n^{-2} h_2^{-2} h_3^{-d} \\
		& \qquad + \abs{\JK_h \ast \mu - \mu}^2 \,,
	\end{align*}
	and 
	\begin{align*}
		& \E \abs{\hat \pi^n_h - \pi}^2 \lesssim \\
		& \qquad n^{-1} h_1^{-1} h_2^{-1} h_3^{-d} + n^{-2} h_1^{-1} h_2^{-2} h_3^{-1-2d} + n^{-2} h_1^{-1} h_3^{-2-2d} \\
		& \qquad + \abs{\HJK_h \ast \pi - \pi}^2 \,.
	\end{align*}

	In the proof of Theorem~\ref{t:minimax-mu} we saw that  
	\begin{equation*}
		\abs{\JK_h \ast \mu - \mu}^2 \lesssim h_2^2 + h_3^{2s_3} \,.
	\end{equation*}
	We may obtain a similar upper bound for $\abs{\HJK_h \ast \pi - \pi}^2$.

	Recall that $\pi = \mu \beta$, where 
	\begin{equation*}
		\beta(t,u,x) = \int_I G(u,v) \int_{\R^d} b(t,x,y) \mu_{t,v}(dy) dv = \int_I G(u,v) \E [b(t, x, X_v(t))] dv \,.
	\end{equation*}
	The boundedness of $b$ ensures the Lipschitz continuity of $\beta$ in variable $u$, while the Lipschitz continuity of $b$ leads to local Hölder $s_3$-continuity of $\beta$ in variable $x$ for $s_3 \in [0,1]$. 
	Moreover, for $0 < t_1 < t_2 < T$, by It\^o's formula we have 
	\begin{align*}
		& \beta(t_2,u,x) - \beta(t_1,u,x) \\
		& = \int_I G(u,v) \E (b(t_2, x,X_v(t_2)) - b(t_1, x,X_v(t_1))) dv \\
		& = \int_I G(u,v) \E \left[ \int_{t_1}^{t_2} (\p_t b_t(x) + \p_y b (x) \beta_{t,v} + \frac{1}{2} \p_y^2 b(x) \tr(\sigma \sigma^T)) (X_v(t)) dt \right] dv \,.
	\end{align*}
	Given the uniform boundedness of $b$, $\p_t b$, $\p_y b$, and $\p_y^2 b$, we have a uniform bound
	\begin{equation*}
		\abs{\beta(t_2,u,x) - \beta(t_1,u,x)} \lesssim \abs{t_2 - t_1} \,,
	\end{equation*}
	and this is further bounded by $O(\abs{t_2 - t_1}^{s_1})$ for $s_1 \in [0,1]$ whenever $t_2 - t_1 < 1$. 
	Then we have $\beta \in \h^{s_1,s_3}(t_0,x_0)$ as well, and so does $\pi = \mu \beta$. 
	Thus
	\begin{equation*}
		\abs{\HJK_h \ast \pi - \pi}^2 \lesssim h_1^{2s_1} + h_2^2 + h_3^{2s_3} \,.
	\end{equation*}

	Joining the above leads to the bound 
	\begin{equation*}
		\E \abs{\hat \beta^n_{h,\kappa} - \beta}^2 \lesssim n^{-1} h_1^{-1} h_2^{-1} h_3^{-d} + n^{-2} h_1^{-1} h_2^{-2} h_3^{-1-2d} + n^{-2} h_1^{-1} h_3^{-2-2d} + h_1^{2s_1} + h_2^2 + h_3^{2s_3} \,.
	\end{equation*}
	Choosing $h_1 = n^{-\frac{s_b}{s_1(2s_b+1)}}$, $h_2 = n^{-\frac{s_b}{2s_b+1}}$, and $h_3 = n^{-\frac{s_b}{s_3(2s_b+1)}}$, we get 
	\begin{equation*}
		\E \abs{\hat \beta^n_{h,\kappa} (t_0,u_0,x_0) - \beta (t_0,u_0,x_0)}^2 \lesssim n^{-\frac{2s_b}{2s_b+1}} \,.
	\end{equation*} 
	Note that the implicit constant depends only on $T,d,\norm{b}_\infty$, the $L^2$ and $L^\infty$ norms of $\HJK$, and the values of $\mu$ in a small neighborhood of $(t_0,u_0,x_0)$. 
	That concludes~\eqref{e:minimax-beta-upper}. 

	\step[Lower bound]
	We construct examples and apply the two-point comparison lemma in a similar way as in the proof of Theorem~\ref{t:minimax-mu}. 

	Let $L > 0$.
	We consider also models with no interactions, so that 
	\begin{equation*}
		\beta(t,u,x) = \int_I G(u) \int_{\R^d} b(t,x) \mu_{t,v}(dy) dv = G(u) b(t,x) \,.
	\end{equation*}
	Pick $b_1(t,x)$ and $G_1(u)$ so that $(b_1,I_{d \times d},G_1,\mu_0) \in \check \a^{s_1,s_3}_{L/2}(t_0,x_0)$. 
	Note that there is no interaction, so the particles $(X^n_1,\dots,X^n_n)$ are independent, with joint law denoted by $\nu_1$. 

	Choose some $\psi \in C_c^\infty(\R \times \R \times \R^d)$ such that 
	\begin{itemize}
		\item $\psi(0,0,0) = 1$, $\norm{\psi}_\infty 1$,
		\item $\int \psi dtdudx = 0$, $\norm{\psi}_2 = 1$,
		\item $\sup_{u \in \R} \norm{\psi(u,\cdot)}_{\h^{s_1,s_3}(t_0,x_0)} < \infty$. 
	\end{itemize}
	For $n \ge 1$ and some small enough $\alpha \in (0,1)$, pick $b_2^n(t,x)$ and $G_2^n(u)$ so that 
	\begin{equation*}
		G_2^n(u) b_2^n(t,x) = G_1(u) b_1(t,x) + \alpha n^{-\frac{1}{2}} (\tau^1_n)^{\frac{1}{2}} (\tau^2_n)^{\frac{1}{2}} (\tau^3_n)^{\frac{d}{2}} \psi(\tau^1_n(t-t_0), \tau^2_n(u-u_0), \tau^3_n(x-x_0)) \,,
	\end{equation*}
	where $(\tau^i_n)_{n \ge 1}$, $i=1,2,3$, are sequences of scalars such that 
	\begin{equation*}
		(\tau^1_n)^{s_1} = (\tau^2_n)^{s_2} = (\tau^3_n)^{\frac{s_3}{d}} = n^{\frac{s_b}{2s_b+1}} \,.   
	\end{equation*}
	With proper choice of parameters $(b_1, G_1, \mu_0)$, small enough $\alpha$ and large enough $n$, the local Hölder continuity of the density follows from classical estimates given in~\cite{BogachevKrylovRocknerShaposhnikov2015}.
	This means we may assume $(b_2^n,G_2^n,I_{d\times d}, \mu_0) \in \check \a^{s_1,s_3}_{L}(t_0,x_0)$. 
	
	Denote by $\nu_2^n$ the joint law of the particles $(X^n_1,\dots,X^n_n)$ derived from the parameters $(b_2^n,G_2^n,I_{d\times d}, \mu_0)$. 
	Following the idea of Lemma 28 in~\cite{MaestraHoffmann2022}, we see that 
	\begin{equation}
		\label{e:lemma28}
		\norm{\nu_1 - \nu_2^n}_{TV}^2  \le \frac{1}{4} \int_0^T \sum_{i=1}^n \E \abs{ G_2^n(\frac{i}{n}) b_2^n(t,X^n_i(t)) - G_1(\frac{i}{n}) b_1(t,X^n_i(t)) }^2 dt \,.
	\end{equation}
	Given the compact support of $\psi$ and local boundedness of $\mu_{t,u}$, \eqref{e:lemma28} is further bounded by 
	\begin{align*}
		\frac{\alpha^2}{4} \int_0^T \tau^1_n \sum_{i=1}^n \tau^2_n (\tau^3_n)^d \norm{\psi(\tau^1_n(t-t_0), \tau^2_n(\frac{i}{n}-u_0), \tau^3_n(\cdot-x_0))}_2^2 dt \lesssim \frac{\alpha^2}{4} < \frac{1}{4}
	\end{align*}
	when $\alpha$ is sufficiently small. 

	On the other hand, 
	\begin{equation*}
		\abs{G_2^n(u_0) b_2^n(t_0,x_0) - G_1(u_0) b_1(t_0,x_0)} = \alpha n^{-\frac{1}{2}} (\tau^1_n)^{\frac{1}{2}} (\tau^2_n)^{\frac{1}{2}} (\tau^3_n)^{\frac{d}{2}} \gtrsim n^{-\frac{s_b}{2s_b+1}} \,.
	\end{equation*}
	Applying Le Cam's two-point comparison method gives the lower bound
	\begin{align*}
		& \inf_{\hat \beta} \sup_{ (b,\sigma,G,\mu_0) \in \check \a^{s_1,s_3}_L (t_0,x_0) } \E \abs{ \hat \beta - \beta(t_0,u_0,x_0) } \\
		\ge & \inf_{\hat \beta} \max_{\tilde \beta \in \{\beta_1, \beta_2^n\}} \E \abs{\hat \beta - \tilde \beta (t_0,u_0,x_0)} \\
		\ge & \frac{1}{2} \abs{\beta_1(t_0,u_0,x_0) - \beta_2^n(t_0,u_0,x_0)} (1 - \norm{\nu_1 - \nu_2^n}_{TV}) \\
		\ge & \frac{1}{4} \abs{G_1(u_0) b_1(t_0,x_0) - G_2^n(u_0) b_2^n(t_0,x_0)} \\
		\gtrsim & n^{-\frac{s_b}{2s_b+1}} \,.
	\end{align*}
\end{proof}

\section*{Acknowledgement}

We express our sincere gratitude to Marc Hoffmann for helpful discussions.

\appendix

\section{Intuitions of Estimators}\label{s:intu-est}

Recall that 
\begin{equation*}
	\beta(t,u,x) = \int_I \int_{\R^d} b(x,y) G(u,v) \mu_{t,v}(dy) dv \,.
\end{equation*}
With Condition~\ref{cd:coeffs}(2) and~\ref{cd:graphon}(2), we may further expand it as 
\begin{equation*}
	\beta(t,u,x) = V(x) \int_I g(u-v) dv + gF \ast \mu_t (u,x) \,,
\end{equation*}
where the convolution here is done on the space $\R \times \R^d$. 

The first term is independent of time $t$, so we have $\p_t \beta = gF \ast \p_t \mu$.
Note that $\p_t$ is a linear operator, and we may approximate it by some finite difference operator 
\begin{equation*}
	D_h f (t_0) \defeq \frac{f(t_0+h) - f(t_0)}{h} = \int_0^T f(t) \frac{\delta_{t_0+h}(dt) - \delta_{t_0}(dt)}{h} \,. 
\end{equation*}
We consider some linear operator $\lin_\phi$ with bounded function $\phi$ that approximates the differential operator, and
\begin{equation*}
	\lin_\phi \beta = (g \otimes F) \ast \lin_\phi \mu \,.
\end{equation*}
Then the deconvolution method gives 
\begin{equation*}
	\f_I \f_{\R^d} \lin_\phi \beta = (\f_I g) (\f_{\R^d} F) (\f_I \f_{\R^d} \lin_\phi \mu) \,,
\end{equation*}
i.e.\ $\t \beta = (\f g) (\f F) \t \mu$. 
Thus it leads to the formula 
\begin{equation*}
	G(u_0,v_0) = \frac{g(u_0-v_0) \norm{F}_2}{g_0 \norm{F}_2} = \frac{\norm{\f_I^{-1} \paren[\big]{ \frac{\t \beta}{\t \mu} } (u=u_0-v_0) }_2}{\norm{\f_I^{-1} \paren[\big]{ \frac{\t \beta}{\t \mu} } (u=0) }_2} 
\end{equation*}
whenever well-defined.

Once we have an estimate of $\mu$ and an estimate of $\beta$, we may plug them into this formula. 
With some additional cutoff factors (to avoid the denominators being too small), we produce the estimator~\eqref{e:estGdef}.

\section{Proofs of Technical Lemmas}\label{s:pf-technical}

\begin{proof}[Proof of Proposition~\ref{pp:mu-L2-bound}]
	Recall that for each $u \in I$, the dynamic of type-$u$ particles is given by $dX_u(t) = \beta(t, u, X_u(t)) dt + \sigma(X_u(t)) d B_u(t)$,
	where $\beta_{t,u} \defeq \beta(t, u, \cdot)$ is Lipschitz on $\R^d$ (this is not hard to verify).
	So $\mu_{t,u}$ uniquely solves the Fokker-Planck equation
	\begin{equation*}
		\p_t \mu_{t,u} = - \grad \cdot (\beta_{t,u} \mu_{t,u}) + \frac{1}{2} \sum_{i,j=1}^d \p_{x_i x_j} ((\sigma \sigma^T)_{ij}  \mu_{t,u}) \,.
	\end{equation*}
	Theorem 7.3.3 in~\cite{BogachevKrylovRocknerShaposhnikov2015} gives a local upper bound
	\begin{equation*}
		\norm{\mu_{t,u}}_{L^\infty(U)} \le C \norm{\mu_{0,u}}_{L^\infty(U)} + C t^{(p-d-2)/2} (1 + \norm{\beta}_{L^p(\mu_{t,u})}^p) 
	\end{equation*}
	for all $p > d+2$ and any bounded open set $U$, whenever the $L^\infty$-norm is well-defined. 
	
	We find that for any $p$, 
	\begin{align*}
		& \int_{\R^d} \abs{\beta(t, u, x)}^p \mu(t, u, dx) \\
		& = \int_{\R^d} \abs{ \int_I \int_{\R^d} b(x, y) G(u, v) \mu(t, v, dy) dv }^p \mu(t, u, dx) \\
		& \lesssim \int_I \int_{\R^d} \abs{\int_{\R^d} b(x, y) \mu(t, v, dy)}^p \mu(t, u, dx) dv \\
		& \le \norm{b}_\infty^p  < \infty \,.
	\end{align*}
	From Condition~\ref{cd:init-data}(1) we know that there exists some $R > 0$ such that $\mu(0,u,x)$ is uniformly bounded by some constant $M$ outside the unit ball $B(0,R)$ for all $u \in I$. 
	We cover $\overline{B(0,R)}^c$ with open balls, so that 
	\begin{equation*}
		\mu(t, u, x) \le C M + C t^{(p-d-2)/2} (1 + \norm{b}_\infty^p) 
	\end{equation*}
	holds for every $t \in [0, T]$, $u \in I$, $\abs{x} > R$. 

	Denote the above upper bound by $C_0$. 
	Then we have for every $t \in [0,T]$ and $u \in I$ that 
	\begin{equation*}
		\int_{\abs{x} > R} \mu(t,u,x)^2 dx \le C_0 \int_{\abs{x} > R} \mu(t,u,x) dx = C_0 \P (\abs{X_u(t)} > R) \le \frac{C_0 \E \abs{X_u(t)}^2}{R^2} \,.
	\end{equation*}
	A classical estimates shows that  $\sup_{t \in [0,T] ,u \in I} \E \abs{X_u(t)}^2 < \infty$, so the above quantity tends to 0 as $R \to \infty$. 
\end{proof}

\begin{proof}[Proof of Leamma~\ref{l:approx-Yu}]
	From the proof of Theorem 3.2 of~\cite{BayraktarChakrabortyWu2023} we see that there exists some constant $C > 0$ such that
	\begin{equation*}
		\E \abs{Y^n_i(t) - Y_{\frac{i}{n}}(t)}^2 \le 6C \max_{1 \le j \le n} \E \abs{X^n_j(t) - X_{\frac{j}{n}}(t)}^2 + \frac{6C}{n^2} 
	\end{equation*}
	for all $i = 1, \dots, n$ and $t \in [0, T]$, which gives the first inequality. 
	The second inequality follows immediately from dominated convergence. 
\end{proof}

\begin{proof}[Proof of Lemma~\ref{l:lip-cont-uv}]
	For the first part, we recall that $\mu$ satisfies the Fokker-Planck equation in the following sense:
	\begin{equation*}
		\p_t \mu_{t,u} + \grad \cdot (\beta_{t,u} \mu_{t,u}) = \frac{1}{2} \sum_{i,j=1}^{d} \p_{ij}^2 ((\sigma \sigma^T)_{ij} \mu_{t,u})  \,.
	\end{equation*}
	For simplicity we work under the (additional) assumption that $\sigma = \sigma_0 I_{d \times d}$ for some constant $\sigma_0 > 0$. 

	Given distinct $u, v \in I$, we have 
	\begin{align*}
		& \p_t (\mu_{t,u} - \mu_{t,v}) \\
		& \qquad = \frac{\sigma_0^2}{2} \lap (\mu_{t,u} - \mu_{t,v}) - \grad \cdot (\beta_{t,u} (\mu_{t,u} - \mu_{t,v})) - \grad \cdot (\mu_{t,v} (\beta_{t,u} - \beta_{t,v})) \,.
	\end{align*}
	Notice that, 
	\begin{equation*}
		\beta_{t,u}(x) = \int_I \int_{\R^d} b(x,y) G(u,u') \mu_{t,u'}(dy) du' \,,
	\end{equation*}
	so there exist some constant $c_1 > 0$ such that 
	\begin{align*}
		\abs{\beta_{t,u}(x) - \beta_{t,v}(x)} & \le \int_I \int_{\R^d} \abs{ b(x,y) (G(u,u') - G(v,u')) } \mu_{t,u'}(dy) du' \\
		& \le c_1 \norm{b}_\infty \int_I \int_{\R^d} \abs{u-v} \mu_{t,u'}(dy) du' \\
		& \le c_1 \norm{b}_\infty \abs{u-v} 
	\end{align*}
	and similarly, $\abs{\grad \cdot \beta_{t,u}(x) - \grad \cdot \beta_{t,v}(x)} \le c_1 \norm{\grad_x \cdot b}_\infty \abs{u-v}$, for any $u, v \in I$. 
	Then we have 
	\begin{align*}
		& \p_t (\mu_{t,u} - \mu_{t,v}) \le \\
		& \qquad \frac{\sigma_0^2}{2} \lap (\mu_{t,u} - \mu_{t,v}) - \grad \cdot \beta_{t,u} (\mu_{t,u} - \mu_{t,v}) - \beta_{t,u} \cdot \grad (\mu_{t,u} - \mu_{t,v}) + C \abs{u-v} \,,
	\end{align*}
	where the constant $C$ depends only on the Lipschitz coefficients of $b$ and $G$. 

	We compare the above inequality with the following differential equation 
	\begin{equation*}
		\p_t \varphi_t = \frac{\sigma_0^2}{2} \lap \varphi_t - \grad \cdot \beta_{t,u} \varphi_t - \beta_{t,u} \cdot \grad \varphi_t \,,
	\end{equation*}
	with initial condition $\varphi_0 = \mu_{0,u} - \mu_{0,v}$. 
	This is a linear homogeneous parabolic equation.
	Thanks to maximum principle, we have 
	\begin{equation*}
		\mu_{t,u} - \mu_{t,v} \le \varphi_t + C t \abs{u-v} \,.
	\end{equation*}
	It remains to bound $\varphi_t$. 

	Consider a time reversal of $\varphi$, namely $\psi(t,x) \defeq \varphi(T-t,x)$. 
	It satisfies the following equation 
	\begin{equation*}
		\p_t \psi_t + (\grad \cdot \beta_{T-t,u}) \psi_t + \beta_{T-t,u} \cdot \grad \psi_t + \frac{\sigma_0^2}{2} \lap \psi_t = 0 \,,
	\end{equation*}
	with terminal condition $\psi_T = \varphi_0$. 
	
	Note that $\grad \cdot \beta_{t,u} = \int_I \int_{\R^d} \grad_x \cdot b(x,y) G(w,w') \mu_{t,w'}(dy) dw'$ is bounded. 
	Then Feynman-Kac formula reads 
	\begin{align*}
		\psi(t,x) & = \tilde \E^{Z_t = x} \left[ \exp \left( \int_t^T \grad \cdot \beta_{T-s,u}(Z_s) ds \right) \psi_T(Z_T) \right]  \\
		& = \tilde \E^{Z_t = x} \left[ \exp \left( -\int_t^T \grad \cdot \beta_{s,u}(Z_{T-s})ds \right) \varphi_0(Z_T) \right] \,,
	\end{align*}
	where $(Z_t)_{t \in [0,T]}$ is a diffusion process with dynamics
	\begin{equation*}
		dZ_t = \beta(T-t,u, Z_t) dt + \sigma d \tilde W_t \,,
	\end{equation*}
	and $\tilde W$ is a $d$-dimensional Brownian motion, under the measure $\tilde \P$. 

	Thus we have
	\begin{equation*}
		\abs{\varphi_t(x)} = \abs{\psi(T-t,x)} \le \norm{\varphi_0}_\infty \exp \paren[\big]{ t \norm{\grad_x \cdot b}_\infty } \le \exp \paren[\big]{ t \norm{\grad_x \cdot b}_\infty } \rho_I(x) \abs{u-v} \,,
	\end{equation*}
	where $\rho_I$ is given by Condition~\ref{cd:init-data}(3). 
	This implies $\mu_{t,u} - \mu_{t,v} \lesssim \abs{u-v}$ at every fixed $x$, where the implicit constant is independent of $x,u,v,t$. 
	
	Similarly, the other direction $\mu_{t,v} - \mu_{t,u}$ produces the same bound. 
	Hence, there exists some constant $C > 0$ such that, for every $t \in [0, T]$, $x \in \R^d$, and $u,v \in I$, it holds that 
	\begin{equation*}
		\abs{\mu(t,u,x) - \mu(t,v,x)} \le C \abs{u-v} \,.
	\end{equation*}
\end{proof}

\section{Reduction of Assumption~\ref{as:FLmu}}\label{s:reduc-FLmu}

Recall the operator $\t = \f_I \f_{\R^d} \lin_\phi$.
Our estimator $\hat G$ requires computing the quantity $\frac{\t \hat \beta}{\t \hat \mu}$. 
If in a set $U \subset \R \times \R^d$ with positive Lebesgue measure we have $\t \mu = 0$, a good estimator $\hat \mu$ would lead to small $\abs{\t \hat \mu}$. 
This might blow up the fraction and keep us away from a good estimation for $G$. 
Therefore, the ad hoc assumption on the nonvanishing property of $\t \mu$ is somehow inevitable in this problem. 
It is also commonly seen in learning sample density with unknown error distribution (see~\cite{Johannes2009} for instance). 

However, the assumption~\ref{as:FLmu} is not a trivial property of $\mu$. 
Given the nonlinearity of the Fokker-Planck equation associated to~\eqref{e:graphon-SDE}, computating the explicit formula for $\t \mu$ is mostly impossible. 
To the best of our knowledge, no explicit solutions for graphon systems~\eqref{e:graphon-SDE} satisfying all our assumptions have been presented.
In this appendix, we study some special cases, where Assumption~\ref{as:FLmu} is reduced to some weaker conditions that are easier to verify.   
We work under the hypothesis of Theorem~\ref{t:main}, and assume $\sigma = I_{d \times d}$ for simplicity. 

\subsubsection*{Case 1: Degeneration to a homogeneous system}

Suppose $I \ni u \mapsto \mu_{0, u} \in \prob(\R^d)$ is constant. 
Corollary 2.3, \cite{Coppini2022note}, states that, given that the map 
\begin{equation*}
	I \ni u \mapsto \int_0^1 G(u, v) dv
\end{equation*}
is also constant (denoted by $\bar g$), the density map $I \ni u \mapsto \mu_u \in \prob(\c_d)$ is also constant, and it solves the classical McKean-Vlasov equation
\begin{equation*}
	\p_t \bar \mu (t, x) = \frac{1}{2} \lap \bar \mu(t, x) - \grad \cdot \left( \bar \mu(t, x) \int_{\R^d} \bar g b(x, y) \bar \mu(t, y) dy \right) \,,
\end{equation*}
where we may view $\bar g b$ as a single quantity. 
In our model $G(u, v) = g(u-v)$, $g$ must be 1-periodic on $[-1, 1]$. 

Notice that for $w \neq 0$, 
\begin{align*}
	\t \mu (w, \xi) & = \int_0^T \phi(t) \int_\R e^{-iwu} \int_{\R^d} e^{-i\xi \cdot x} \mu(t, u, x) dx du dt \\
	& = \int_0^T \phi(t) \int_0^1 e^{-iwu} \int_{\R^d} e^{-i\xi \cdot x} \mu(t, u, x) dx du dt \\
	& = \int_0^T \phi(t) \frac{1 - e^{-iw}}{iw} \f \bar \mu_t (\xi) dt \\
	& = \frac{1 - e^{-iw}}{iw} \f \mathcal{L} \bar \mu (\xi) \,.
\end{align*}
Note that $1 - e^{-iw} = 0$ if and only if $w = 2 \pi k$ for $k \in \Z$, which means it is nonzero $dw$-a.e. 
Then our Assumption~\ref{as:FLmu} reduces to Assumption 16 (on the solution $\bar \mu$) in~\cite{MaestraHoffmann2022}.

\subsubsection*{Case 2: Denegeration to a finite graph}

Define the \emph{degree} of an index $u$ with respect to a subset $J \subset I$ by 
\begin{equation*}
	\deg_J (u) = \int_J G(u, v) dv \,,
\end{equation*}
and 
\begin{equation*}
	\deg (u) = \int_0^1 G(u, v) dv \,.
\end{equation*}

Consider the partition $I = \bigcup_{j=1}^m I_j$, and denote by $[u_0]$ the subset $I_j \ni u_0$. 
Assume the degree on each part of the partition is constant, i.e., for $u_0, u_1, u_2 \in I$,  we have 
\begin{equation}
	\label{e:constant-degree-on-partition}
	\deg_{[u_0]}(u_1) = \int_{[u_0]} G(u_1, v) dv = \int_{[u_0]} G(u_2, v) dv = \deg_{[u_0]} (u_2) \,,
\end{equation}
whenever $[u_1] = [u_2]$. 
An example is $G(u, v) = g(u-v)$, where $g$ is defined on $\R$, supported on $[-1, 1]$, and $\frac{1}{m}$-periodic on $[-1, 1]$. 
Then we take $I_j = (\frac{j-1}{m}, \frac{j}{m}]$ for every $j = 1, \dots, m$, so that 
\begin{equation*}
	\deg_{I_j} (u_1) = \int_{\frac{j-1}{m}}^{\frac{j}{m}} g(u_1 - v) dv = \int_{u_1 - \frac{j-1}{m}}^{u_1 - \frac{j}{m}} g(v) dv = \int_{0}^{\frac{1}{m}} g(v) dv \,.
\end{equation*}

Assume further that the initial data map $u \mapsto \mu_{0,u} \in \prob(\R^d)$ is also constant on each $I_j$. 
Theorem 2.1 in~\cite{Coppini2022note} tells us that the map $u \mapsto \mu \in \prob(\c_d)$ is constant in each part of the partition, that is, $\mu_{u_1} = \mu_{u_2}$ whenever $[u_1] = [u_2]$. 

We set $\mu^{(j)} (t, x) = \mu_{t,u}(x)$, for $u \in I_j$, $j = 1, \dots, m$, 
and set $D_{jk} = d_{I_k} (u)$ for $u \in I_j$, which is well-defined due to~\eqref{e:constant-degree-on-partition}. 
Then they satisfy a family of coupled equations: 
\begin{equation*}
	\p_t \mu^{(j)} = \frac{1}{2} \lap \mu^{(j)} - (D \vec{1})_j \grad \cdot (\mu^{(j)} V) - \grad \cdot \left(\mu^{(j)}  \sum_{k=1}^{m} D_{jk} F \ast \mu^{(k)} \right) \,, \qquad j = 1, \dots, m \,.
\end{equation*}
where $D = (D_{jk})$ is treated as an $m \times m$ matrix. 
Assume that all $D_{jk}$ are equal to some constant $d_0 > 0$. 
We see that 
\begin{equation*}
	\p_t \mu^{(j)} = \frac{1}{2} \lap \mu^{(j)} - m d_0 \grad \cdot (\mu^{(j)} V)  - \grad \cdot \left( \mu^{(j)} d_0 F \ast \paren[\big]{ \sum_{k=1}^{m} \mu^{(k)}  } \right) \,.
\end{equation*}
Let $\bar \mu = \frac{1}{m} \sum_{j=1}^{m} \mu^{(j)} $ (it is indeed $\int_I \mu_u du$ in this situation), then it solves 
\begin{equation}
	\label{e:pde-for-nu}
	\p_t \bar \mu = \frac{1}{2} \lap \bar \mu - m d_0 \grad \cdot (\bar \mu V) - m d_0 \grad \cdot (\bar \mu F \ast \bar \mu) \,.
\end{equation}
Note that $\bar \mu \in \prob(\R^d)$, and we may define
\begin{equation*}
	\bar b(t, x, \bar \mu_t) \defeq m d_0 (V(x) + \int_{\R^d} F(x - y) \bar \mu_t(dy)) \,.
\end{equation*}
Then~\eqref{e:pde-for-nu} is the associated Fokker-Planck equation for the mean-field diffusion process 
\begin{equation*}
	d U_t = \bar b(t, U_t, \bar \mu_t) dt + d \bar B_t \,, \qquad U_0 \sim \bar \mu_0 \,,
\end{equation*}
and $\bar \mu$ is the density of $U$. 

Now we may compute 
\begin{align*}
	\t \mu (w, \xi)  = \sum_{k=1}^{m} c_k(w)\f \mathcal{L} \mu^{(k)} (\xi) \,,
\end{align*}
where in this case 
\begin{equation*}
	c_k(w) \defeq \int_{I_k} e^{-iwv} dv = \begin{cases}
		\frac{e^{- \frac{iwk}{m}} (e^{\frac{iw}{m}}-1)}{iw} & w \neq 0 \\
		\frac{1}{m} & w = 0
	\end{cases} .
\end{equation*}
We focus on the part where $w \neq 0$, where $\t \mu \neq 0$ if and only if 
\begin{equation*}
	\f \mathcal{L} \paren[\big]{ \sum_{k=1}^{m} e^{-\frac{iwk}{m}}  \mu^{(k)} } \neq 0 \,.
\end{equation*}
Set
\begin{equation*}
	\rho_w \defeq \sum_{k=1}^{m} e^{-\frac{iwk}{m}} {\mu^{(k)} } \,.
\end{equation*}
Notice that for every $w \in \R$, $\rho_w$ solves the linear differential equation 
\begin{equation}
	\label{e:pde-for-rhow}
	\p_t \rho = \frac{1}{2} \lap \rho - m d_0 \grad \cdot (\rho (V + F \ast \bar \mu_t)) \,.
\end{equation}
Its canonical diffusion process has the following dynamics
\begin{equation*}
	d R_t = \bar \beta (t, R_t) dt +  d B_t \,,
\end{equation*}
where $\bar \beta (t, x) = \bar b(t, x, \bar \mu_t)$. 
Then Assumption~\ref{as:FLmu} now reduces to Assumption 16 in~\cite{MaestraHoffmann2022} on $\rho_w$ for almost every $w \in [0, 2\pi m]$. 

\begin{remark}
	In the most general case, $(\mu_u)_{u \in I}$ are the solutions to a system of infinitely many fully coupled nonlinear differential equations. 
	There are no explicit formulae either for $\f_I \mu_t (w) = \int_I e^{-iwu} \mu_{t,u} du$ to fit into a condition involving only the operator $\f_{\R^d} \lin$. 
	However, each $\rho_w$ in Case 2 is now the solution to some linear equation (though the coefficients involve $\bar \mu_t$, which solves some other equation and could be seen as a known quantity).
	The assumption becomes much milder in this sense, and that is the main reduction in Case 2. 
\end{remark}

\bibliographystyle{abbrv}
\bibliography{refs}

\end{document}